\documentclass[12pt]{amsart}
\usepackage{amssymb,amsmath,amsthm,amscd,enumerate,verbatim,extarrows}
\usepackage[colorlinks=true,linkcolor=black,citecolor=black]{hyperref}
\usepackage[all]{xy}
\usepackage{xcolor,cleveref}
\usepackage{fullpage}

\numberwithin{equation}{section}

\newtheorem{thm}{\bf Theorem}[section]
\newtheorem{lem}[thm]{\bf Lemma}

\newtheorem{cor}[thm]{\bf Corollary}
\newtheorem{prop}[thm]{\bf Proposition}

\theoremstyle{definition}
\newtheorem{dfn}[thm]{Definition}
\newtheorem{ex}[thm]{Example}
\newtheorem{rem}[thm]{Remark}
\newtheorem{quest}[thm]{Question}

\newtheorem*{thm*}{Theorem}
\newtheorem{lem*}[thm]{\bf Lemma}

\DeclareMathOperator{\Ann}{Ann}
\DeclareMathOperator{\Ass}{Ass}
\DeclareMathOperator{\chara}{char}
\DeclareMathOperator{\codepth}{codepth}
\DeclareMathOperator{\codim}{codim}

\DeclareMathOperator{\depth}{depth}
\DeclareMathOperator{\embdim}{embdim}
\DeclareMathOperator{\Ext}{Ext}

\DeclareMathOperator{\glind}{glld}
\DeclareMathOperator{\gr}{gr}
\DeclareMathOperator{\grade}{grade}

\DeclareMathOperator{\lind}{ld}
\DeclareMathOperator{\linp}{lin}
\DeclareMathOperator{\Min}{Min}

\DeclareMathOperator{\pd}{pd}

\DeclareMathOperator{\reg}{reg}
\DeclareMathOperator{\rk}{rank}

\DeclareMathOperator{\Tor}{Tor}

\newcommand{\wh}{\widehat}
\newcommand{\ot}{\otimes}

\newcommand{\FF}{{\mathbb F}}
 
\newcommand{\ZZ}{{\mathbb Z}}

\newcommand{\QQ}{{\mathbb Q}}

 \newcommand{\Fc}{{\mathcal F}}

\def\mm{{\mathfrak m}}

\def\nn{{\mathfrak n}}

\newcommand{\ovl}{\overline}

\newcommand{\wht}{\widehat}

\newcommand{\bsx}{\boldsymbol x}

\title[Koszul property and $g$-stretched rings]{Koszul property and finite linearity defect\\ over $g$-stretched local rings}

\author[D.V. Kien]{Do Van Kien}
\address{Department of Mathematics, Hanoi Pedagogical University 2, Xuan Hoa, Phuc Yen, Vinh Phuc, Vietnam}
\email{dovankien@hpu2.edu.vn}
\author[H.D. Nguyen]{Hop D. Nguyen}
\address{Institute of Mathematics, Vietnam Academy of Science and Technology, 18 Hoang Quoc Viet, 10307 Hanoi, Vietnam}
\email{ndhop@protonmail.com}
\begin{document}
\begin{abstract}
The linearity defect is a measure for the non-linearity of minimal free resolutions of modules over noetherian local rings. A tantalizing open question due to  Herzog and Iyengar asks whether a noetherian local ring $(R,\mm)$ is Koszul if its residue field $R/\mm$ has a finite linearity defect. We provide a positive answer to this question when $R$ is a Cohen--Macaulay local ring of almost minimal multiplicity with the residue field of characteristic zero. The proof depends on the study of noetherian local rings $(R,\mm)$ such that $\mm^2$ is a principal ideal, which we call \emph{$g$-stretched} local rings. The class of $g$-stretched local rings subsumes stretched artinian local rings studied by Sally, and generic artinian reductions of Cohen--Macaulay local rings of almost minimal multiplicity. An essential part in the proof of our main result is a complete characterization of one-dimensional complete $g$-stretched local rings. Beside partial progress on Herzog--Iyengar's question, another consequence of our study is a numerical characterization of all $g$-stretched Koszul rings, strengthening previous work of Avramov, Iyengar, and \c{S}ega. 
\end{abstract}
\maketitle

\section{Introduction}
\label{sect_intro}

Starting from classical results like Hilbert's syzygy theorem and Auslander--Buchsbaum--Serre's characterization of regular local rings, gaining an understanding of the structure of free resolutions over a local ring (or standard graded algebras over a field $k$)  has proved to be of fundamental importance. Among standard graded $k$-algebras, Koszul algebras have emerged as adequate playgrounds for studying graded free resolutions thanks to work of Avramov, Eisenbud, and Peeva \cite{AE92}, \cite{AP01}. Inspired by work of Eisenbud et al. \cite{EFS}, Herzog and Iyengar \cite{HIy} introduced a natural generalization of Koszul algebras to the local situation, via the notion of \emph{linearity defect} (see \Cref{sect_prelim} for details). Following \cite{HIy}, we say that a noetherian local ring (or a standard graded $k$-algebra) $(R,\mm,k)$ is \emph{Koszul} if $k$ has a linear free resolution over $\gr_\mm(R)$, equivalently, if $k$ has linearity defect $\lind_R k=0$. Combining results of \cite{AE92, AP01, HIy}, we have:
\begin{thm}[Avramov, Eisenbud, Herzog, Iyengar, Peeva]
\label{thm_Koszul_charact}
Let $(R,\mm)$ be a standard graded $k$-algebra. For a finitely generated graded $R$-module $M$, let $\reg_R M$ and $\lind_R M$ be the regularity and the linearity defect of $M$ over $R$, resp. The following are equivalent:
\begin{enumerate}[\quad \rm (1)]
\item $R$ is a Koszul algebra, i.e. $\reg_R k=0$;
\item $\reg_R k<\infty$;
\item $\reg_R M<\infty$ for any finitely generated graded $R$-module $M$;
\item $\lind_R k=0$;
\item $\lind_R k<\infty$.
\end{enumerate}
\end{thm}
For recent surveys related to free resolutions and Koszul algebras, we refer to \cite{Avr98,CDR,PS}. 

The finiteness of the linearity defect is known to have strong consequences: By \cite[Proposition 1.8]{HIy}, if $\lind_R M<\infty$, then the Poincar\'e series of $M$ is a rational function with denominator depending only on $R$, and independent of $M$. On the other hand, not much is known about the classification of noetherian local rings (or graded algebras) such that $\lind_R M<\infty$ for every finitely generated (graded) $R$-module $M$. Such rings are said to be \emph{absolutely Koszul}, and the reader may consult, for example, \cite{AMS20, CINR, Ng15, Ng18}, for some recent studies on them. In particular, the following natural question on the finiteness of the linearity defect remains open.
\begin{quest}[{Herzog and Iyengar \cite[Question 1.14, p.\,162]{HIy}}]
\label{quest_HIy}
Let $(R,\mm,k)$ be a noetherian local ring. Assume that $\lind_R k<\infty$. Does it follow that $\lind_R k=0$, i.e. $R$ is Koszul?
\end{quest}
Recall that if $(R,\mm, k)$ is a Cohen--Macaulay local ring of multiplicity $e(R)$ and embedding dimension $\mu(\mm):=\dim_k(\mm/\mm^2)$, then it always satisfies the Abhyankar  inequality \cite{Ab}: $e(R)\ge \codim(R) + 1$, where $\codim(R):=\mu(\mm)-\dim R$ denotes the \emph{embedding codimension} of $R$. For this reason,  $R$ is said to have \emph{minimal multiplicity} (respectively, \emph{almost minimal multiplicity}) if $e(R)= \codim(R) + 1$ (resp. $e(R)= \codim(R) + 2$).  

Beside the graded case mentioned in \Cref{thm_Koszul_charact}, \Cref{quest_HIy} has a positive answer in either of the following cases. Below, we let $r(R):=\dim_k \Ext^{\depth R}_R(k,R)$ be the \emph{type} of $R$.
\begin{enumerate}
\item $\mm^3=0$ \cite{Se}, or more generally $\mm^4=0$ \cite{AM17};
\item $R$ is a complete intersection \cite{My19,Se};
\item $R$ is Cohen--Macaulay and either has minimal multiplicity or has almost minimal multiplicity of type $r(R)\le \codim(R)-1$ \cite[Proposition 6.1.8]{CDR} (indeed, such a ring is Koszul);
\item $R$ is Cohen--Macaulay and Golod \cite{My19}.
\end{enumerate}

In this paper, we are interested in the Koszul property of $R$ and the finiteness of the linearity defect of the residue field $k$.  Our main result is a positive answer to \Cref{quest_HIy} for Cohen--Macaulay local rings of almost minimal multiplicity whose residue fields contain $\mathbb{Q}$.
\begin{thm}[= \Cref{thm_HerzogIyengar_AlmostMinMult}]
\label{thm_main_HerzogIyengar}
Let $(R,\mm,k)$ be a Cohen--Macaulay local ring of almost minimal multiplicity. Assume that $\chara k=0$. If $\lind_R k<\infty$, then the ring $R$ is Koszul.
\end{thm}

The proof of \Cref{thm_main_HerzogIyengar} consists of three steps:
\begin{enumerate}
\item \textbf{Step 1}: Reduction to dimension at most 1. The main ingredient is the solution of Sally's conjecture due to  Rossi--Valla \cite{RV96} and independently Wang \cite{W97}.
\item \textbf{Step 2}: Solution to the case $\dim R=0$. The four main ingredients are the classification of $g$-stretched local rings of dimension 1, Elias--Valla's structure theorem for artinian stretched local rings, the theory of weak Koszul filtration, and  our previous work \cite{Ng15, NgV16} on the behavior of linearity defect along short exact sequences.
\item \textbf{Step 3}: Solution to the case $\dim R=1$. The proof uses the same strategy as the solution of Step 2, combined with reduction to the case of Step 2.
\end{enumerate}
In Step 1, we use the fact that if $R$ is a Cohen--Macaulay local ring of almost minimal multiplicity with $\dim R\ge 2$, then the associated graded ring has almost maximal depth: There is an inequality $\depth \gr_\mm(R)\ge \dim R-1$. This was conjectured by Sally and proved by Wang \cite{W97} and  Rossi--Valla \cite{RV96}. In particular, if moreover $k$ is infinite, then we can choose an $R$-regular element $x\in \mm$ such that $x+\mm^2\in \mm/\mm^2$ is $\gr_\mm(R)$-regular, and we can reduce the dimension by replacing $R$ by $R/(x)$. Therefore the main case of \Cref{thm_main_HerzogIyengar} is when $\dim R\le 1$.

Before going to the details of Steps 2 and 3, let us introduce $g$-stretched rings. Observe that $R$ is an artinian local ring of almost minimal multiplicity if and only if $\mm^3=0$ and the equality $\mu(\mm^2)=1$ holds\footnote{Indeed, an artinian ring $R$ has minimal multiplicity if and only if $\ell(R)=\mu(\mm)+2$. The chain $\ell(R)=\ell(R/\mm)+\ell(\mm/\mm^2)+\ell(\mm^2)=1+\mu(\mm)+\ell(\mm^2)$ translates the last equation to $\ell(\mm^2)=1$. Nakayama's lemma implies that this is equivalent to  $\mm^3=0$ and $\mu(\mm^2)=1$.}. More generally, let $(R,\mm, k)$ be a  Cohen--Macaulay local ring with infinite residue field, $Q$ be a parameter ideal of $R$ generated by elements in $\mm\setminus \mm^2$. If $R$ satisfies the inequality $e(R)\le \codim(R) + 2$, then the above observation implies that the ideal $(\mm/Q)^2$ of the artinian reduction $R/Q$ is principal. Based on the these observations, it is natural to study noetherian local ring whose square of the maximal ideal is principal. We call such rings {\it $g$-stretched}, where ``g" stands for ``generalized". Note that artinian $g$-stretched local rings are exactly the class of \emph{stretched} local rings considered by Sally \cite{Sal79}, and Elias--Valla \cite{ElV08}, among others.

Step 2 is accomplished by the following 
\begin{thm}[See \Cref{thm_HIy_gstretched_long}]
\label{thm_HIy_gstretched_short}
Let $(R,\mm,k)$ be a noetherian local ring. Assume that $R$ is a $g$-stretched ring and $\chara(k)=0$. Then $R$ is Koszul if and only if $\lind_R k<\infty$.
\end{thm}
This theorem says that \Cref{quest_HIy} has a positive answer for $g$-stretched rings whose residue fields contain $\QQ$. The proof of \Cref{thm_HIy_gstretched_short} has four main ingredients. Observe that Krull's principal ideal theorem implies that any $g$-stretched local ring has dimension at most 1. It is not hard to see that any Cohen--Macaulay 1-dimensional $g$-stretched local ring is regular, namely a discrete valuation ring; see \Cref{prop_1dimCM_gstretched}. The first ingredient in the proof of \Cref{thm_HIy_gstretched_short} is the following structure result for one-dimensional $g$-stretched rings, that are quotients of regular local rings. This result provides a positive answer to a question proposed to us by N. Matsuoka.

\begin{thm}[= \Cref{thm_Matsuoka}]
\label{thm_Matsuoka_intro}
Let $(S,\nn)$ be a regular local ring of dimension $d\ge 1$, and $I\subseteq \nn^2$ an ideal such that $S/I$ is a $g$-stretched ring of dimension 1. Then $I=Q\nn$, where $Q$ is generated by a regular sequence of $d-1$ elements in $\nn\setminus \nn^2$.
\end{thm}
Answering Matsuoka's question was a starting point for this work; our proof uses a double induction: one on $d$ and another noetherian induction on $S/I$. 

As an application of \Cref{thm_Matsuoka_intro}, we prove the following strengthening of a result due to Avramov, Iyengar and \c{S}ega \cite[Theorem 4.1]{AIS}. Note that Avramov, Iyengar and \c{S}ega only considered rings with $\mu(\mm^2)\le 1$ and $\mm^3=0$, while we consider more generally rings with $\mu(\mm^2)\le 1$.

%\newpage
\begin{thm}[= \Cref{thm_Koszul}]
Let $(R,\mm,k)$ be a $g$-stretched local ring. Then the following are equivalent:
\begin{enumerate}[\quad \rm (1)]
\item $R$ is Koszul;
\item Either $\dim R=1$, or $R$ is artinian, $\mm^3=0$, and $r(R) \le \codim(R)-1$ unless $\mm^2=0$.
\end{enumerate}
\end{thm}

The second ingredient in the proof of \Cref{thm_HIy_gstretched_short} is a structure result for stretched local rings due to Elias and Valla \cite[Theorem 3.1]{ElV08}. Using the latter, we deduce the harder part of \Cref{thm_HIy_gstretched_short}, when $R$ is Gorenstein.

\begin{thm}[= \Cref{thm_HIy_artinGor}]
Let $(R,\mm,k)$ be a $g$-stretched artinian local ring.  Assume that $\chara(k)=0$. If $\lind_R k<\infty$ and $R$ is Gorenstein, then $R$ is a Koszul ring.
\end{thm}
The proof of the last theorem depends on the theory of \emph{weak Koszul filtration}, which we develop in \Cref{sect_weakKoszulfiltr}. In fact, a key and (at least to us) surprising observation in the proof of \Cref{thm_HIy_artinGor}, is that any stretched Gorenstein local ring has a weak Koszul filtration, even though it might not be Koszul. The notion of weak Koszul filtration for local rings subsumes that of Koszul filtration for graded $k$-algebras considered by Conca, Trung, and Valla \cite{CTV}. While this notion is less restrictive than the local Koszul filtration proposed in \cite{Ng15}, it is flexible enough for controlling of the linearity defect: We prove in \Cref{lem_weakfiltr_Tormap} that any weak Koszul filtration induces various short exact sequences with \emph{Tor-vanishing} maps, and the existence of such short exact sequences allow for very precise bounds on the linearity defect. Hence the third ingredient in the proof of \Cref{thm_HIy_gstretched_short} is the theory of weak Koszul filtration, and the fourth is various bounds for linearity defect along short exact sequences induced by such filtration, developed in \cite{Ng15, NgV16}.

Step 3 in the proof  of \Cref{thm_main_HerzogIyengar} is concerned with the case where $R$ has dimension $1$. The rough idea is  taking advantage of the special structure of an artinian reduction $R/(x)$, which is stretched, to retrieve information about $R$ itself. The arguments for Step 3 are largely parallel to those for Step 2, employing the same ingredients, but some care has to be taken since in contrast to the case $\dim R\ge 2$, when $\dim R=1$, some linearity defect information may be lost in the passage from $R$ to $R/(x)$.

\noindent\textbf{Organization.} After the background \Cref{sect_prelim}, we reduce \Cref{thm_main_HerzogIyengar} to dimension at most 1 in \Cref{sect_dim1_reduction}. We formally introduce the weak Koszul filtration and prove its key property in \Cref{sect_weakKoszulfiltr}. \Cref{sect_g-stretched} is devoted to the proof of \Cref{thm_Matsuoka_intro}. The main result of \Cref{sect_charact_Koszulness} is the numerical characterization of $g$-stretched Koszul rings, strengthening previous work of Avramov, Iyengar, and \c{S}ega. \Cref{sect_dim0_case} deals with the zero-dimensional case of \Cref{thm_main_HerzogIyengar}: We deduce it from the more general \Cref{thm_HIy_gstretched_short} (see \Cref{thm_HIy_gstretched_long}). In Section \ref{sect_dim1_case}, we prove our main result, \Cref{thm_main_HerzogIyengar} after resolving its dimension one case. We end the paper by discussing some related open questions.

\section{Preliminaries}
\label{sect_prelim}
We recall the notion of linearity defect which Herzog and Iyengar introduced in \cite{HIy}, based in the notion of the linear part of a minimal free resolution. The linear part appeared in work of Herzog et al.\,\cite[Section 5]{HSV83} and Eisenbud et al.\,\cite{EFS}. Let $(R, \mm, k)$ be a noetherian local ring, and $M$ a finitely generated $R$-module. Let $\gr_\mm(R)=\oplus_{i\ge0}\mm^i/\mm^{i+1}$ be the associated graded ring of $R$ with respect to the $\mm$-adic filtration. Let the
minimal free resolution of $M$ be
$$\FF:\cdots\xrightarrow{\partial}F_i\xrightarrow{\partial}F_{i-1}\xrightarrow{\partial}\cdots\xrightarrow{\partial}F_1\xrightarrow{\partial}F_0\to0.$$
Since $\FF$ is minimal, it admits a filtration $\cdots\subseteq\Fc^i\FF\subseteq\Fc^{i-1}\FF\subseteq\cdots\subseteq\Fc^0\FF = \FF$ as follows:
$$\Fc^i\FF:\cdots\to F_j\to F_{j-1}\to\cdots\to F_i\to \mm F_{i-1}\to\cdots\to \mm^{i-1}F_1\to \mm^iF_0\to 0.$$
The associated graded complex
$$\linp^R\FF=\bigoplus_{i\ge 0}\frac{\Fc^i\FF}{\Fc^{i+1}\FF}$$
is called {\it the linear part} of $\FF$. Clearly, $\linp^R\FF$ is a complex of graded $\gr_\mm(R)$-modules with $(\linp^R\FF)_i = (\gr_\mm F_i)(-i)$. The {\it linearity defect} of $M$ is defined to be
$$\lind_R M = \sup\left\{ i : H_i(\linp^R\FF)\neq 0\right\}.$$
By convention, if $M = 0$, we set $\lind_R M = 0$. We say that $M$ is a {\it Koszul module} if $\lind_R M = 0$. A local ring $R$ is said to be {\it Koszul} if its residue field $k$ is Koszul as a module over $R$. 

We say that an $R$-linear map $\phi: M\to P$ is \emph{Tor-vanishing} if the induced maps $\Tor^R_i(k,\phi)$ have zero images for all $i\in \ZZ$. Using \c{S}ega's resolution-free characterization of the linearity defect in \cite[Theorem 2.2]{Se}, we can prove the following statements.
\begin{prop}[{\cite[Corollary 2.10]{Ng15}} and {\cite[Proposition 4.3]{NgV16}}]
\label{prop_ld_exactseq}
Consider any short exact sequence  $0\to M\xlongrightarrow{\phi} P\to N \to 0$ of finitely generated $R$-modules. 
\begin{enumerate}[\quad \rm (1)]
\item There are inequalities:
\begin{enumerate}[\quad \rm(i)]
\item $\lind_R N \le \min \{\max\{\pd_R P+1,\lind_R M+1\}, \max\{\lind_R P, \pd_R M+1\} \}$,
\item $\lind_R P \le \min \{\max\{\pd_R M+1,\lind_R N\}, \max\{\lind_R M, \pd_R N\}\}$,
\item $\lind_R M \le \min \{\max\{\lind_R N-1, \pd_R P\},\max\{\pd_R N, \lind_R P\} \}$.
\end{enumerate}
\item If $P$ is free then $\lind_R M=\lind_R N-1$ if $\lind_R N\ge 1$ and $\lind_R M=0$ if $\lind_R N=0$.
\item If $\phi$ is Tor-vanishing, then there are inequalities:
\begin{align*}
\lind_R M &\le \max\{\lind_R P, \lind_R N-1\},\\
\lind_R P &\le \max\{\lind_R M, \lind_R N\},\\
\lind_R N &\le \max\{\lind_R M+1,\lind_R P\}.
\end{align*}
\end{enumerate}
\end{prop}

\begin{thm}[{\cite[Theorem 3.5 and its proof]{Ng15}}]
\label{thm_small_inclusion}
Let $0\longrightarrow M\xlongrightarrow{\phi} P \longrightarrow N \longrightarrow 0$ be a short exact sequence of finitely generated $R$-modules where
\begin{enumerate}[\quad\rm(i)]
\item $P$ is a Koszul module;
\item $M\subseteq \mm P$.
\end{enumerate}
Then $\phi$ is Tor-vanishing, and there are inequalities $\lind_R N-1 \le \lind_R M \le \max\{0,\lind_R N-1\}$. In particular, $\lind_R N=\lind_R M+1$ if $\lind_R M\ge 1$ and $\lind_R N\le 1$ if $\lind_R M=0$. 

Moreover, $\lind_R N=0$ if and only if $M$ is a Koszul module and $M\cap \mm^{s+1}P=\mm^sM$ for all $s\ge 0$.
\end{thm}

\begin{rem}
\label{rem_ldge1}
Let $(R,\mm,k)$ be a noetherian local ring. If $\lind_R k\le 1$, then $\lind_R k=0$. Indeed, as $R$ is a Koszul $R$-module, the exact sequence
\[
0\to \mm \to R \to k \to 0
\]
and \Cref{thm_small_inclusion} imply that $\lind_R \mm \le \max\{0,\lind_R k-1\}=0$. But since $\mm \cap \mm^{s+1}R=\mm^s\mm$ for all $s\ge 0$, the same result implies that $\lind_R k=0$, as desired.
\end{rem}

A local ring $R$ is \emph{absolutely Koszul} if every finitely generated $R$-module has a finite linearity defect; see \cite{CINR, HIy} for more details. Let 
\[
\glind R=\sup \{\lind_R M: M \, \, \text{a finitely generated $R$-module}\}
\]
denote the \emph{global linearity defect} of $R$. To study the absolutely Koszul property, the passage to completion rings, that is allowed by the following result, might be useful.
\begin{prop}
\label{prop.passtocompletion}
Let $(R,\mm)$ be a noetherian local ring, and $(\wht{R},\wht{\mm})$ its $\mm$-adic completion. Then $R$ is Koszul \textup{(}respectively, absolutely Koszul\textup{)} if and only if $\wht{R}$ is so. Moreover, we always have equalities $\lind_R k=\lind_{\wht{R}}k$ and $\glind R=\glind \wht{R}$.
\end{prop}
The proof uses the following lemma, which was used for the demonstration of \cite[Theorem 1.1]{NgV18}. This lemma might be of independent interest, hence we recall its proof for completeness.
\begin{lem}
\label{lem.reductiontofinitelength}
Let $(R,\mm)$ be a noetherian local ring, $I\subseteq \mm$ an ideal and $M$ a finitely generated $R$-module. Then there exists an integer $N$ \textup{(}depending only on $M$ and $I$\textup{)} such that for every integer $n\ge N$,  there is an equality
\[
\lind_R(M/I^nM) =\max\{\lind_R M, \lind_R (I^nM)+1\}.
\]
\end{lem}

\begin{proof}
We use a result of Avramov \cite[Corollary A.4]{Avr78}, saying that there exits an integer $N_1\ge 1$, depending only on $I$ and $M$, such that for all $n\ge N_1$, $\Tor^R_i(k,I^nM) \to \Tor^R_i(k,M)$ is the zero map for all $i$, i.e. $I^nM\to M$ is Tor-vanishing. Another application of the same result for the pair $(I, I^{N_1}M)$ shows that there exists an integer $N_2\ge 1$ such that for all $n\ge N_2$, the map  $I^{n+N_1}M \to I^{N_1}M$ is Tor-vanishing. 

Take $n\ge N:=N_1+N_2$. Let $F, G, H$ be minimal free resolutions of $I^nM$, $I^{N_1}M$, and  $M$, respectively. Let $\phi_1: F\to G, \phi_2: G\to H$ be liftings of the natural maps $I^nM \to I^{N_1}M$ and $I^{N_1}M \to M$, respectively. Then $\phi_1(F)\subseteq \mm G, \phi_2(G)\subseteq \mm H$. Denoting $\phi=\phi_2\circ \phi_1: F\to H$, then $\phi$ is a lifting of the natural map $I^nM \to M$ and $\phi(F)\subseteq \mm^2H$. 

Consider the short exact sequence
\[
0\to I^nM \to M \to M/I^nM \to 0.
\]
The mapping cone $W=H\oplus F[-1]$ of $\phi: F\to H$ is then a minimal free resolution of $M/I^nM$. Moreover, the condition $\phi(F)\subseteq \mm^2H$ implies the direct sum decomposition of complexes of $\gr_\mm R$-modules
\[
\linp^R W \cong \linp^R H \oplus (\linp^R F)[-1].
\]
This implies 
\[
\lind_R(M/I^nM) =\max\{\lind_R M, \lind_R (I^nM)+1\},
\]
and concludes the proof.
\end{proof}
\begin{proof}[Proof of \Cref{prop.passtocompletion}]
 Let $M$ be any finitely generated $R$-module. Since $\gr_\mm R\cong \gr_{\wht{\mm}}\wht{R}$, applying  \cite[Lemma 3.1]{NgV16}, we get 
\[
\lind_R M=\lind_{\wht{R}}(M\otimes_R \wht{R}).
\]
In particular $\lind_R k=\lind_{\wht{R}}k$ and $R$ is Koszul if and only if $\wht{R}$ is so. Moreover,  $\glind R\le \glind \wht{R}$, and if $\wht{R}$ is absolutely Koszul, then so is $R$.

Let $N$ be any finitely generated $\wht{R}$-module. By \Cref{lem.reductiontofinitelength}, there exists an integer $N\ge 1$ such that for all $n\ge N$,
\[
\lind_{\wht{R}}(N/\wht{\mm}^nN)=\max\{\lind_{\wht{R}} N, \lind_{\wht{R}} (\wht{\mm}^n N)+1\}.
\]
Fix such an integer $n$, and let $U=N/\wht{\mm}^nN$. We claim that $U$ is a finitely generated $R$-module, and $U\otimes_R \wht{R} \cong U$ as $\wht{R}$-modules. Indeed, since $N$ is finitely generated over $\wht{R}$, $U$ is finitely generated over $\wht{R}/\wht{\mm}^n$. Since the natural map $R/\mm^n \to \wht{R}/\wht{\mm}^n$ is an isomorphism, we deduce that $U$ is a finitely generated $R/\mm^n$-module. In particular, it is a finitely generated $R$-module. Since $\mm^n U=0$, clearly
\[
U\otimes_R \wht{R} \cong \wht{U} \cong U.
\]
This gives a chain
\[
\lind_{\wht{R}} N \le \lind_{\wht{R}} U = \lind_{\wht{R}} (U\otimes_R \wht{R}) =\lind_R U.
\]
Thus $\glind \wht{R}\le \glind R$, and if $R$ is absolutely Koszul, then so is $\wht{R}$. The desired conclusions follow.
\end{proof}

Let $\bsx=x_1,\ldots,x_e$ be a minimal system of generators of $\mm$. Denote by $K^R$ the Koszul complex $K(\bsx;R)$ on $\bsx$. The homology $H_j(K^R)$ of $K^R$ are known to be independent of the choice of the sequence $\bsx$. Denote $\codepth R=e-\depth R$. For a finitely generated $R$-module $M$, the formal power series
\[
P^R_M(t)=\sum_{i=0}^\infty \dim_k \Tor^R_i(k,M) t^i \in \ZZ[[t]]
\]
is called the \emph{Poincar\'e series} of $M$. Serre proved the (coefficient-wise) inequality
\[
P^R_k(t) \preccurlyeq \frac{(1+t)^e}{1-\sum_{j=1}^{\codepth R}\dim_k H_j(K^R)t^{j+1}}.
\]

\begin{dfn}[See {\cite[Section 5]{Avr98}}]
\label{dfn.Golod}
A noetherian local ring $(R,\mm,k)$ is called \emph{Golod} if we have the equality of formal power series
\[
P^R_k(t) = \frac{(1+t)^e}{1-\sum_{j=1}^{\codepth R}\dim_k H_j(K^R)t^{j+1}}.
\]
\end{dfn}
Since completion does not change the embedding dimension $e$, the depth, the Poincar\'e series $P^R_k(t)$, and the rank of the Koszul homology $H_j(K^R)$, we deduce from \Cref{dfn.Golod} that 
\begin{lem}
\label{lem.Golod.passgetocompletion}
A noetherian local ring $(R,\mm)$ is Golod if and only if its $\mm$-adic completion $(\wht{R},\wht{\mm})$ is so.
\end{lem}

\section{Reduction of \Cref{thm_main_HerzogIyengar} to dimension at most 1}
\label{sect_dim1_reduction}
The following allows us to reduce the statement of \Cref{thm_main_HerzogIyengar} to the case $\dim R\le 1$.
\begin{prop}
\label{prop_dim1_reduction}
Let $(R,\mm,k)$ be a Cohen--Macaulay local ring of almost minimal multiplicity of dimension $\dim R\ge 2$. Assume that $k$ is infinite. Then there exists an $R$-regular element $x\in \mm$ such that $x+\mm^2\in \mm/\mm^2$ is $\gr_\mm(R)$-regular. In particular, letting $\ovl{R}:=R/(x)$ and $\ovl{\mm}:=\mm/(x)$, the following statements hold simultaneously:
\begin{enumerate}[\quad \rm (1)]
\item $(\ovl{R},\ovl{\mm},k)$ is Cohen--Macaulay of almost minimal multiplicity of dimension $\dim R-1$;
\item $\lind_R k=\lind_{\ovl{R}}k$;
\item $R$ is Koszul if and only if $\ovl{R}$ is so.
\end{enumerate}
\end{prop}

The notions of superficial elements (in $\mm$ with respect to $R$) and filter-regular elements (with respect to a graded ring) will be crucial to our discussions.
\begin{dfn}
(1) Given a local ring $(R,\mm)$, we say $x\in \mm$ is a \emph{superficial element of $\mm$ with respect to $R$} (\emph{superficial element of $\mm$} for short),  if there exists a non-negative integer $c$ such that 
\[
(\mm^{n+1}:x)\cap \mm^c=\mm^n \quad \text{for every $n\ge c$}.
\]
A sequence $x_1,\ldots,x_s\in \mm$ is called a \emph{superficial sequence of $\mm$} if for every $i=1,\ldots,s$, the image of $x_i$ in $\mm/(x_1,\ldots,x_{i-1})$ is a superficial element with respect to $R/(x_1,\ldots,x_{i-1})$.

(2) Let $A=\bigoplus_{n\in \ZZ}A_n$ be a graded ring. We say that a homogeneous element $a\in A_d$ is \emph{filter-regular with respect to $A$} if for all $n\gg 0$, the map $A_n\xrightarrow{\cdot a} A_{n+d}$ is injective.
\end{dfn}
It is straightforward to check that an element $x\in \mm$ is superficial (with respect to $R$) if and only if $x+\mm^2 \in \mm/\mm^2$ is filter-regular with respect to $\gr_\mm R$.

\begin{rem}
\label{rem_hyperplanesection}
(1) By \cite[Lemma 8.5.4]{SH06}, if $x\in \mm$ is superficial and $\depth R\ge 1$, then $x$ is also $R$-regular. If moreover the residue field $k=R/\mm$ is infinite and $\depth R\ge 1$, then there exists $x\in \mm$ which is both a superficial element of $\mm$ and  $R$-regular; see \cite[Corollary 8.5.9]{SH06} for details.

(2) Assume that $k$ is infinite and either $\dim R\ge 2$ or $\depth R\ge 1$. Let $x\in \mm$ be a superficial element that is not contained in any minimal prime ideal of $R$ (whose existence is guaranteed by \cite[Corollary 8.5.9]{SH06}), and  $\ovl{R}=R/(x)$. As mentioned above, if $\depth R\ge 1$, then the superficiality of $x$ implies that $x$ is $R$-regular. Therefore, using  \cite[Proposition 11.1.9(2)]{SH06} for $M=R$, there is an equality $e(R)=e(\ovl{R})$. In particular, if $R$ has almost minimal multiplicity, then so is $\ovl{R}$.
\end{rem}

\begin{proof}[Proof of \Cref{prop_dim1_reduction}]
By \cite{RV96, W97}, $\depth \gr_\mm(R)\ge \dim R-1\ge 1$. So by \Cref{rem_hyperplanesection}, there exists $x\in \mm$ which is both a superficial element of $\mm$ and  $R$-regular; in fact, it suffices to choose $x\in \mm$ such that $x+\mm^2$ is $\gr_\mm(R)$-regular. Now $(\ovl{R},\ovl{\mm},k)$ is Cohen--Macaulay of almost minimal multiplicity of dimension $\dim R-1$ thanks to \Cref{rem_hyperplanesection}. We have $\lind_R \ovl{R}=0$, so $\lind_R k=\lind_{\ovl{R}}k$ by \cite[Theorem 5.2]{Ng15}. The last equality implies that $R$ is Koszul if and only if $\ovl{R}$ is so, and the proof is completed.
\end{proof}

\begin{rem}
(1) The argument of \Cref{prop_dim1_reduction} cannot be used to reduce the statement of \Cref{thm_main_HerzogIyengar} to the artinian case: It may happen for a one-dimensional Cohen--Macaulay local ring of almost minimal multiplicity that $\depth \gr_\mm(R)=0$. Letting $k=\QQ$, the ring 
$$
R=k[[t^4,t^5,t^{11}]]\cong k[[x,y,z]]/(y^3-xz,x^4-yz,x^3y^2-z^2)
$$ 
furnishes a concrete example. We have $R$ is 1-dimensional Cohen--Macaulay of $\codim(R)=2$. Furthermore, per Macaulay2 \cite{GS96} computations,
\begin{align*}
\gr_\mm(R)&\cong k[x,y,z]/(xz,yz,y^4,z^2),\\
 e(R)&=e(\gr_\mm(R))=4=\codim(R)+2,
\end{align*}
  so $R$ has minimal multiplicity. On the other hand, $\depth \gr_\mm(R)=0$. Nevertheless $R$ is not Koszul and the proof of \Cref{prop_dim1_nonGor} below shows that $\lind_R k=\infty$.

(2) While given $\dim R=1$, we cannot use \Cref{prop_dim1_reduction} to immediately reduce  \Cref{thm_main_HerzogIyengar} to the artinian case, the proof of \Cref{thm_main_HerzogIyengar} in that case \textbf{does} take advantage of the special structure of the artinian reduction of $R$, as we will see in \Cref{sect_dim1_case}.
\end{rem}

\section{Weak Koszul filtration}
\label{sect_weakKoszulfiltr}
Inspired by the notion of Koszul filtration for graded algebras \cite{CTV} (originating from work of Bruns, Herzog, and Vetter \cite[Page 10]{BHV}), the second author introduced in \cite{Ng15} an analogous tool to detect Koszulness of local rings. 
\begin{dfn}[{\cite[Definition 4.1]{Ng15}}]\label{dfnKoszulfil}
Let $(R,\mm)$ be a local ring. Let $\mathcal F$ be a collection of
ideals. We say that $\mathcal F$ is a \emph{Koszul filtration} of $R$ if the following simultaneously hold:
\begin{enumerate}
	\item[(1)] $(0),\mm \in  \mathcal F$ 
\item[(2)] for every ideal $I\in  \mathcal F$ and all $s\ge 1$, we have $I\cap \mm^{s+1} = \mm^sI$, and
\item[(3)] for every ideal $I\neq (0)$ of $\mathcal F$, there exist a finite filtration $(0) = I_0\subset I_1\subset\cdots\subset I_n = I$ and elements $x_j\in \mm$, such that for each $j =1,\ldots, n, I_j\in \mathcal F, I_j = I_{j-1} + (x_j)$ and $I_{j-1} : x_j\in \mathcal F$.
\end{enumerate}
\end{dfn}

By \cite[Theorem 4.3(i)-(ii)]{Ng15}, if $R$ has a Koszul filtration $\Fc$, then $R$ is Koszul. Moreover, for any $I$ in a Koszul filtration $\Fc$ of $R$, $\lind_R(R/I)=0$, and hence by \Cref{prop_ld_exactseq}(2), $\lind_R I=0$. Thus $I$ is a Koszul $R$-module for any ideal $I\in \Fc$.

\begin{ex}
\label{ex_regularlocal}
Let $(R,\mm, k)$ be a regular local ring, and $x_1,\ldots,x_d$ be a minimal generating set of $\mm$. Then $R$ has a Koszul filtration $\mathcal F=\{(0),(x_1),\ldots,(x_1,\ldots,x_d)\}$. Indeed, it suffices to check for each $1\le j\le d$ and $s\ge 1$ the following equality where $I=(x_1,\ldots,x_j)$:
\begin{equation}
\label{eq_intersect_Ipowersofm}
I\cap \mm^{s+1}=\mm^s I.
\end{equation}
This is a consequence of \Cref{lem_regseq_intersect}(2) below. Hence every regular local ring is Koszul.  Moreover, any ideal generated by a subset of a minimal generating set of $\mm$ is a Koszul $R$-module.
\end{ex}
\begin{lem}
\label{lem_regseq_intersect}
In a noetherian local ring $(R,\mm)$, let $x_1,\ldots,x_m,y_1,\ldots,y_n$ be an $R$-regular sequence in $\mm$, where $m,n\ge 1$. Denote $I=(x_1,\ldots,x_m), J=(y_1,\ldots,y_n)$. Then for all integers $s,t\ge 1$, $i\ge 0$, there are equalities
\begin{enumerate}[\quad \rm (1)]
\item $I^s \cap J^t=I^sJ^t$,
\item $I^s\cap (I+J)^{s+i}=I^s(I+J)^i$.
\end{enumerate}
\end{lem}
\begin{proof}
(1) We have to show that $\Tor^R_1(R/I^s, R/J^t)=0$. Since $I$ is generated by a regular sequence, $I^i/I^{i+1}$ is a free $(R/I)$-module for each $i\ge 0$, by \cite[Theorem 1.1.8]{BH}. Similar thing holds for $J^i/J^{i+1}$. Using  exact sequences of the form
\[
0\to I^i/I^{i+1} \to R/I^{i+1} \to R/I^i \to 0,
\]
and induction, we reduce to the case $s=1$, namely showing that $\Tor^R_1(R/I, R/J^t)=0$. Similar arguments show that it suffices to assume $s=t=1$, and to show  $\Tor^R_1(R/I, R/J)=0$. The last statement is true since $I$ is generated by an $(R/J)$-regular sequence; see, e.g., \cite[Exercise 1.1.12]{BH}. This concludes the proof of (1).

(2) Denote $L=I+J$. As $L^{s+i}=I^sL^i+J^{i+1}L^{s-1}$, we get the first equality in the chain
\begin{align*}
I^s \cap L^{s+i} &= I^s \cap (I^sL^i+J^{i+1}L^{s-1}) = I^sL^i+ I^s \cap J^{i+1}L^{s-1}\\
                 &\subseteq I^sL^i+ I^s\cap J^{i+1} =I^sL^i +I^sJ^{i+1}= I^sL^i\subseteq I^s \cap L^{s+i}.
\end{align*}
The third equality holds by (1). Thus equalities hold from left to right, and the proof is concluded.
\end{proof}

\begin{dfn}[Weak Koszul filtration]
\label{dfn_wKf}
Let $(R,\mm)$ be a noetherian local ring. A family of \emph{proper} ideals $\Fc$ of $R$ is called a \emph{weak Koszul filtration} if the following conditions hold:
\begin{enumerate}[\quad \rm (WF1)]
 \item $(0),\mm \in \Fc$;
 \item for every $I\in \Fc$, $I\cap \mm^2=\mm I$;
 \item for every nonzero $I\in \Fc$, there exist $x\in \mm$ and $J\in \Fc$ such that $I=J+(x)$ and $J:x\in \Fc$.
\end{enumerate}
\end{dfn}

The following example shows that the notion of weak Koszul filtration strictly subsumes that of Koszul filtration.
\begin{ex}
It is not hard to see that if $(R,\mm)$ is a standard graded $k$-algebra, and $\Fc$ is a Koszul filtration in the sense of Conca et al. \cite{CTV}, then $\Fc$ is a weak Koszul filtration. It is also clear that any Koszul filtration in the sense of \cite[Definition 4.1]{Ng15} is a weak Koszul filtration. 

However, the converse is not true: There are local rings with a weak Koszul filtration, but are not Koszul. For example, let $R=k[[x,y]]/(xy,x^3-y^2)$. Then $R$ has the following weak Koszul filtration (by abuse of notation, we denote the residue class in $R$ of an element $a\in k[[x,y]]$ by $a$ itself):
\[
\{(0),(x), (y), \mm=(x,y)\}.
\]
Indeed,  we can check the following equalities, where the last two hold since $x,y\notin \mm^2$:
\begin{align*}
(0):x &= (y), &\mm^2 =(x^2),\\
(0):y&=(x),  & (x)\cap \mm^2 =\mm(x),\\
(x):y &= \mm, & (y)\cap \mm^2 =\mm (y).
\end{align*}
On the other hand, $R$ is not a Koszul ring since $\gr_\mm(R)\cong k[x,y]/(xy,y^2,x^4)$ is not a Koszul $k$-algebra, as the latter is not even defined by quadrics.
\end{ex}
\begin{rem}
\label{rem_WKF_genbylinearform}
Let $\Fc$ be a weak Koszul filtration of $R$, and $I\neq (0)$ an ideal in $\Fc$. We claim that there exist elements $y_1,\ldots,y_p\in \mm\setminus \mm^2$ such that:
\begin{enumerate}
 \item $I$ is minimally generated by $y_1,\ldots,y_p$, and
 \item for each $1\le i \le p$, we have $(y_1,\ldots,y_i) \in \Fc$ and $(y_1,\ldots,y_{i-1}):y_i\in \Fc$.
\end{enumerate}
Indeed, by (WF3), there exists an element $y_1\in \mm$ and $J_1\in \Fc$ such that $I=(y_1)+J_1$ and $J_1:y_1\in \Fc$. By definition, $\Fc$ does not contain the unit ideal, so $y_1\notin J_1$, which in turn implies $J_1\neq I$. Nakayama's lemma then implies that $y_1\notin \mm I=I\cap \mm^2$ (per (WF2)),  consequently $y_1\in \mm\setminus \mm^2$.

If $J_1=(0)$, then $I=(y_1)$, and we are done. If $J_1\neq (0)$, since $J_1\in \Fc$, we can continue the above argument to get a filtration of ideals $I=J_0\supsetneq J_1\supsetneq J_2 \supsetneq \cdots$ and elements $y_i\in \mm$ such that for each $i\ge 1$:
\begin{enumerate}
 \item $J_{i-1}=(y_i)+J_i$,
 \item $J_i, J_i:y_i\in \Fc$,
 \item $y_i\in \mm \setminus \mm^2$.
\end{enumerate}
Since $J_{i-1}$ strictly contains $J_i\supseteq (y_1,\ldots,y_{i-1})$, we have  $y_i\notin (y_1,\ldots,y_{i-1})$. The strictly increasing chain 
\[
(y_1) \subsetneq (y_1,y_2) \subsetneq (y_1,y_2,y_3) \subsetneq \cdots
\]
has to terminate, so there exists some $p\ge 1$ such that $J_p=(0)$. For this $p$, we get that $y_1,\ldots,y_p$ is a minimal generating set of $I$, $(y_1,\ldots,y_i)\in \Fc$ and $(y_1,\ldots,y_{i-1}):y_i\in \Fc$ for each $1\le i\le p$. This is the desired conclusion.
\end{rem}
A key property of any weak Koszul filtration is that it induces short exact sequences with Tor-vanishing maps.
\begin{lem}
\label{lem_weakfiltr_Tormap}
Let $(R,\mm,k)$ be a noetherian local ring and $\Fc$ a weak Koszul filtration of $R$. Let $(0) \neq I \in \Fc$ be an ideal and let $x\in \mm$ and $J\in \Fc$ be such that $I=J+(x)$ and $J:x\in \Fc$. \textup{(}The existence of $x$ and $J$ is guaranteed by the definition of weak Koszul filtration.\textup{)} Then:
\begin{enumerate}[\quad \rm (1)]
 \item For all $i\ge 1$, the natural map $\Tor^R_i(R/\mm^2,R/I) \to \Tor^R_i(R/\mm,R/I)$ is zero. Equivalently, the connecting map $\Tor^R_i(R/\mm,R/I) \to \Tor^R_{i-1}(\mm/\mm^2,R/I)$ induced by the exact sequence $0\to \mm/\mm^2\to R/\mm^2\to R/\mm \to 0$, is injective.
 \item The natural map $R/(J:x) \xrightarrow{\cdot x} R/J$ is Tor-vanishing.
\end{enumerate}
\end{lem}
\begin{proof}
The subsequent argument is similar to \cite[Proof of Theorem 4.3]{Ng15}.

The conclusion is clear if $\mm=(0)$, i.e. $R$ is a field. So we assume that $\mm \neq (0)$.

Arguing as in \Cref{rem_WKF_genbylinearform}, for any nonzero $I\in \Fc$, there exist a finite filtration $(0)=I_0 \subset I_1 \subset \cdots \subset I_n=I$ and elements $x_j\in \mm$ for $j=1,\ldots,n$ such that for all $1\le j \le n$, $I_j\in \Fc, I_j=I_{j-1}+(x_j)$ and $I_{j-1}:x_j\in \Fc$. Moreover, the argument in \Cref{rem_WKF_genbylinearform} indicates that we may choose $I_{n-1}=J$ and $x_n=x$.

We prove by induction on $i$ and $j$ the following statements.
\begin{enumerate}
 \item[(S1)] For all $i\ge 0$ and all $0\le j\le n$, the map
\[
\Tor^R_{i+1}(R/\mm,R/I_j)\longrightarrow \Tor^R_i(\mm/\mm^2,R/I_j)
\] 
is injective.
\item[(S2)] For all $i\ge 0$ and all $1\le j\le n$, the map $\Tor^R_i(R/(I_{j-1}:x_j),k)\to \Tor^R_i(R/I_{j-1},k)$ is zero. 
\end{enumerate}
Choosing $j=n$ in (S1) and (S2), we get the desired conclusions, as $I_n=I$, $I_{n-1}=J$, $x_n=x$.

Firstly, assume that $i=0$. Note that $I_j\cap \mm^2 \subseteq \mm I_j$ by (WF2) in \Cref{dfn_wKf}. So the natural map 
\[
\Tor^R_1(R/\mm^2,R/I_j)=(I_j\cap \mm^2)/(\mm^2 I_j)\longrightarrow \Tor^R_1(R/\mm,R/I_j)=I_j/(\mm I_j)
\]
is zero. Hence (S1) is true for $i=0$ and all $0\le j\le n$. Clearly (S2) is true for $i=0$ and all $1\le j\le n$, since the image of $R/(I_{j-1}:x_j)$ is contained in $x_j(R/I_{j-1})$.

Now assume that $i\ge 1$. Note that (S1) is true for $j=0$ as $I_0=(0)$, so it is harmless to assume that $j\ge 1$. Note also that (S2) is true for $i\ge 1$ and $j=1$, as $\Tor^R_i(R/I_0,k)=\Tor^R_i(R,k)=0$.

Denote $y=x_j, L=I_{j-1}$ so that $I_j=L+(y)$. We have an exact sequence
\[
0\longrightarrow R/(L:y) \xlongrightarrow{\cdot y} R/L\longrightarrow R/I_j \longrightarrow 0.
\]
We have $\mm/\mm^2\cong k^h$ for some positive integer $h$. Let $1\le \ell \le i$ be an integer. Consider the commutative diagram with obvious connecting and induced maps
\begin{displaymath}
\xymatrix{\Tor^R_{\ell}(R/\mm,R/(L:y)) \ar[r]^{\rho_{\ell}} \ar[d] & \Tor^R_{\ell}(R/\mm,R/L) \ar[d]^{\tau_{\ell}}\\
\Tor^R_{\ell-1}(\mm/\mm^2,R/(L:y))  \ar[r]^{\rho^h_{\ell-1}}     & \Tor^R_{\ell-1}(\mm/\mm^2,R/L).
}
\end{displaymath}
 Observe that $\tau_{\ell}$ is injective: if $\ell < i$ then this follows from the induction on $i$ for (S1), while if $\ell=i$ then, recalling that $L=I_{j-1}$, this follows from the induction hypothesis on $j$ for (S1). Since $\ell-1\le i-1$, by the induction hypothesis for (S2), $\rho_{\ell-1}$ is the zero map. Inspecting the diagram,  $\rho_{\ell}$ is zero as well. Hence $\rho_j$ is the trivial map for all $1\le j\le i$. In particular, $\rho_i$ is the trivial map, namely (S2) is true for $i$.

Now consider the diagram with obvious connecting and induced maps
\begin{displaymath}
\xymatrix{
         &                                                      & \Tor_{i-1}(\mm/\mm^2,R/(L:y)) \ar[d]^{\rho_{i-1}}\\
0 \ar[r] & \Tor_i(R/\mm,R/L) \ar[d] \ar[r]^{\alpha^i_2} & \Tor_{i-1}(\mm/\mm^2,R/L)\ar[d] \\
         & \Tor_i(R/\mm,R/I_j) \ar[d] \ar[r]^{\alpha^i_3} & \Tor_{i-1}(\mm/\mm^2,R/I_j)\ar[d]\\ 
0 \ar[r] & \Tor_{i-1}(R/\mm,R/(L:y))  \ar[r]^{\alpha^{i-1}_1}        & \Tor_{i-2}(\mm/\mm^2,R/(L:y)) }          
\end{displaymath}
By the hypothesis of the induction on $j$ (respectively, on $i$) for (S1), the maps $\alpha^i_2$ (resp.~ $\alpha^{i-1}_1$) are injective. We know from the previous paragraph that $\rho_{i-1}$ is the zero map. Hence by a snake lemma argument, $\alpha^i_3$ is also injective. Hence (S1) is also true, finishing the induction.
\end{proof}

The following lemma allows us to lift a weak Koszul filtration from a quotient ring modulo a regular element to the original ring. In fact, we will prove a slightly more general statement.
\begin{lem}[Canonical lifting of a weak Koszul filtration]
\label{lem_weakfiltr_lifting}
Let $(R,\mm)$ be a noetherian local ring, and $y\in \mm\setminus \mm^2$ an element. Let $\ovl{\Fc}$ be a weak Koszul filtration of $\ovl{R}:=R/(y)$. Consider the collection of ideals
\[
\Fc = \{I \mid (y)\subseteq I, I/(y) \in \ovl{\Fc}\}  \, \cup \, \{(0)\}.
\]
Assume moreover that at least one of the following conditions is satisfied:
\begin{enumerate}[\quad \rm (a)]
 \item $\Ann_R(y)=(0)$, i.e. $y$ is $R$-regular;
 \item $y\in \Ann_R(y)$ \textup{(}namely $y^2=0$\textup{)}, and $\Ann_R(y)/(y) \in \ovl{\Fc}$.
\end{enumerate}
Then $\Fc$ is a weak Koszul filtration of $R$.
\end{lem}

\begin{proof}
Let $\ovl{I}\neq (0)$ be an ideal of $\ovl{\Fc}$. By \Cref{rem_WKF_genbylinearform}, we may choose $y_1,\ldots,y_p\in \mm\setminus (\mm^2+(y))$ such that
\begin{enumerate}
 \item $\ovl{I}$ is minimally generated by $\ovl{y_1},\ldots,\ovl{y_p}$, and,
 \item for each $1\le i \le p$, we have $(\ovl{y_1},\ldots,\ovl{y_i}) \in \ovl{\Fc}$ and $(\ovl{y_1},\ldots,\ovl{y_{i-1}}):\ovl{y_i}\in \ovl{\Fc}$.
\end{enumerate}
Let $I=(y,y_1,\ldots,y_p)$ be the $R$-ideal uniquely determined by $\ovl{I}$, since $y\in I$ and $I/(y)=\ovl{I}$. Then $\Fc$ consists of ideals of the form $I=(y,y_1,\ldots,y_p)$ for each $(0)\neq \ovl{I} \in \ovl{\Fc}$, and $(0)$. Moreover, for each $(0)\neq \ovl{I} \in \ovl{\Fc}$, the ideal $I=(y,y_1,\ldots,y_p)$ has the property that:
\begin{enumerate}[\quad \rm (i)]
 \item $I$ is minimally generated by $y,y_1,\ldots,y_p$;
 \item for each $1\le i\le p$, both $(y,y_1,\ldots,y_i)$ and $(y,y_1,\ldots,y_{i-1}):y_i$ are $\in \Fc$.
\end{enumerate}
Note that since $(0) \in \ovl{\Fc}$, we have by definition that $(y)\in \Fc$.

Clearly $(1)\notin \Fc$. Now we verify that $\Fc$ satisfies the conditions (WF1)--(WF3).

(WF1): Since $\ovl{\Fc}$ contains the zero and maximal ideals of $\ovl{R}$, $\Fc$ contains the corresponding ideals of $R$.

(WF2): Take $I\in \Fc$, we check that $I\cap \mm^2=\mm I$. There is nothing to do if $I=(0)$. Thus we may assume that $y\in I$ and $I/(y)\in \ovl{\Fc}$. If $I/(y)=(0)$ then $I=(y)$. In this case
\[
(y)\cap \mm^2=\mm y
\]
since per hypothesis $y\notin \mm^2$.

We may therefore assume that $I/(y)\neq (0)$. As above, we deduce an equality $I=(y,y_1,\ldots,y_p)$, where $p\ge 1$, and $\ovl{I}\cap \ovl{\mm}^2=\ovl{\mm}\ovl{I}$ per the hypothesis on $\ovl{\Fc}$.  In other words,
\[
\mm I+(y)=I\cap (\mm^2+(y))=(y)+I\cap \mm^2.
\]
The second equality holds since $y\in I$.

The containment $\mm I\subseteq I\cap\mm^2$ is clear. For the reverse containment, take $a\in I\cap \mm^2$. The last display yields $a\in \mm I+(y)$, so we may write $a=by+c_1y_1+\cdots+c_py_p$, where $b\in R, c_1,\ldots,c_p\in \mm$. Since $a\in \mm^2$, we deduce
\[
by=a-c_1y_1-\cdots-c_py_p\in \mm^2.
\]
As $y\notin \mm^2$, we get $b\in \mm$. Therefore $a=by+c_1y_1+\cdots+c_py_p\in \mm I$, as desired.

(WF3): Take $(0)\neq I\in \Fc$. If $I/(y)=(0)$, then $I=(y)$. In this case $(0):_R y=\Ann_R(y) \in \Fc$ thanks to the hypothesis that either (a) or (b) is fulfilled. Therefore (WF3) is satisfied.

Assume that $I/(y)\neq (0)$. As above, there is an equality $I=(y,y_1,\ldots,y_p)$, where $p\ge 1$, and the following two conditions are satisfied:
\begin{itemize}
 \item $I$ is minimally generated by $y,y_1,\ldots,y_p$;
 \item for each $1\le i\le p$, both $(y,y_1,\ldots,y_i)$ and $(y,y_1,\ldots,y_{i-1}):y_i$ are $\in \Fc$.
\end{itemize}
Choose $J=(y,y_1,\ldots,y_{p-1})$. We get that $J\in \Fc$, $I=J+(y_p)$ and $J:y_p\in \Fc$. Hence again (WF3) is satisfied. The proof is concluded.
\end{proof}

\section{$g$-stretched rings} \label{sect_g-stretched} 
\quad Throughout this section, let $(R,\mm)$ be a noetherian local ring. 
\begin{dfn}
We say that $R$ is a {\it $g$-stretched ring} (where ``g'' stands for ``generalized''),  if $\mu_R(\mm^2)\le 1$, that is, $\dim_k(\mm^2/\mm^3)\le 1$.
\end{dfn}
By Krull's principal ideal theorem, if $R$ is $g$-stretched, them $\dim R\le 1$. Following Sally \cite{Sal79}, we say that $R$ is \emph{stretched} if it is $g$-stretched and of dimension $0$. Our aim in this section is to study in depth the dimension 1 case.
\begin{prop}
\label{prop_1dimCM_gstretched}
Suppose $R$ is a $g$-stretched Cohen--Macaulay ring of dimension 1. Then $R$ is a DVR.
\end{prop}
\begin{proof}
Let $\mm$ be minimally generated by $a_1,\ldots,a_n$. Since $\mm^2$ is principal, $\mm^2=(a_ia_j)$ for some $1\le i\le j\le n$. By \cite[Proposition 1.2.10]{BH}, we have the second equality in the chain 
$$
1=\grade(\mm,R)=\grade(\mm^2,R)=\grade((a_ia_j),R).
$$ 
Hence $a_i$ is an $R$-regular element. For each $1\le l\le n$, since $a_ia_l \in (a_ia_j)$, we get $a_l \in (a_j)$. Hence $n=1$ and $\mm=(a_1)$, which yields the desired conclusion.
\end{proof}
\quad Note that a one-dimensional $g$-stretched ring need not be a DVR, for instance $R=k[[x,y]]/(xy,x^2)$. This is also a typical example for such rings by the following theorem, in which we completely describe $\mm$-adically complete one-dimensional $g$-stretched rings. 
\begin{thm} 
\label{thm_Matsuoka}
Let $(S,\nn)$ be a regular local ring of dimension $d\ge 1$, and $I\subseteq \nn^2$ an ideal such that $S/I$ is a one-dimensional $g$-stretched ring. Then $I=Q\nn$, where $Q$ is generated by a regular sequence of $d-1$ elements in $\nn\setminus \nn^2$.
\end{thm}
In the proof, we will proceed by a double induction using the following lemma.

\begin{lem}
\label{lem_forbidden_containment}
Let $(S,\nn)$ be a regular local ring of dimension $d\ge 2$, $Q_1,Q_2$ distinct proper ideals each of which is generated a regular sequence of $d-1$ elements in $\nn\setminus \nn^2$. Then for any element $g\in \nn^2$,  $\nn^2$ is \textbf{not} contained in $Q_1\cap Q_2+(g)$.
\end{lem}
\begin{proof}
We can assume that $Q_1=(y_1,\ldots,y_{d-1})$ and $\nn=(y_1,\ldots,y_d)$, where $y_i\in \nn \setminus \nn^2$. The ring $S/Q_1$ is regular local of dimension 1 with uniformizer $y_d+Q_1$, hence it is a principal ideal domain. The ideal  $(Q_1+Q_2)/Q_1 \neq S/Q_1$ is therefore principal, say generated by $y_d^s+Q_1$ for some $s\ge 1$. This implies $Q_1+Q_2=Q_1+(y_d^s)=(y_1,\ldots,y_{d-1},y_d^s)$.

Let $Q_2=(z_1,\ldots,z_{d-1})$, where the $z_i$s form a regular sequence of elements in $\nn\setminus \nn^2$. We can write
\[
z_i=\sum_{j=1}^{d-1}a_{ij}y_j+b_{id}y_d^s, a_{ij}, b_{id}\in S.
\]
\textbf{Claim 1:} Not all of $b_{1d},\ldots,b_{d-1,d}$ are in $\nn$.

Assume the contrary, then $z_i\in Q_1+\nn y_d^s$. Hence
\[
Q_1+(y_d^s)=Q_1+Q_2\subseteq Q_1+\nn y_d^s.
\]
Nakayama's lemma then implies $Q_1+(y_d^s)=Q_1$, namely $y_d^s\in Q_1$. This is a contradiction.

Thus we can assume, say $b_{1n} \notin \nn$. It is harmless to assume that $b_{1n}=1$.

For $2\le i\le d-1$, denote $z'_i=z_i-b_{in}z_1 \in Q_1$. We note that $z'_i\notin \nn^2$ for all $2\le i\le d-1$. Indeed, assume otherwise. Choose $z_d$ such that $z_1,\ldots,z_{d-1},z_d$ minimally generated $\nn$. Since $z_i-b_{in}z_1\in \nn^2$, working modulo $(z_j: j\neq i)$, we get $z_i\in (z_i^2)$ in a suitable regular local ring of dimension 1 with uniformizer $z_i$. This is impossible.

We also see that $z_1,z'_2,\ldots,z'_{d-1}$ form a regular sequence in $\nn\setminus \nn^2$, that generates $Q_2$. So it is harmless to assume that $z_2,\ldots,z_{d-1}\in Q_1$. Working modulo $(z_2,\ldots,z_{d-1})$, we see that $Q_1=(z_2,\ldots,z_{d-1},y)$ for some $y\in \nn\setminus \nn^2$ such that $z_2,\ldots,z_{d-1},y$ is a regular sequence. For simplicity, denote $z=z_1,L=(z_2,\ldots,z_{d-1})$, so $Q_2=L+(z), Q_1=L+(y)$.

\textbf{Claim 2:} $Q_1\cap Q_2=L+(yz)$.

It suffices to show that $Q_1\cap Q_2 \subseteq L+(yz)$. Take $w\in Q_1\cap Q_2$. As $w\in Q_2$, we can assume that $w\in (z)$, namely $w=az, a\in S$. Since $Q_1+Q_2=(z_2,\ldots,z_{d-1},y,z)$ is $\nn$-primary and $\dim S=d$, $z_2,\ldots,z_{d-1},y,z$ is a system of parameters for $S$. But $S$ is Cohen--Macaulay, hence $z_2,\ldots,z_{d-1},y,z$ is a regular sequence. Now $az\in Q_1=(z_2,\ldots,z_{d-1},y)$, so $a\in Q_1$. Therefore $w=az\in Q_1z \subseteq L+(yz)$, as desired.

Finally, assume that $\nn^2\subseteq Q_1\cap Q_2+(g)$ for some $g\in \nn^2$. Then by Claim 2,
\[
\nn^2 \subseteq L+(yz,g).
\]
Working modulo $L$, we reduce to the case $(S,\nn)$ is a regular local ring of dimension 2 such that $\nn^2\subseteq (yz,g)$ for some $y,z\in \nn, g\in \nn^2$. Thus $\nn^2=(yz,g)$. Letting $k=S/\nn$, this implies $\dim_k \nn^2/\nn^3 \le 2$. This is impossible, since $A=\gr_\nn S$ is a standard graded two-dimensional polynomial ring $k[u,v]$, so $\dim_k A_2=\dim_k \nn^2/\nn^3=3$. This contradiction finishes the proof.
\end{proof}

We are now ready to prove Theorem \ref{thm_Matsuoka}.
\begin{proof}[Proof of Theorem \ref{thm_Matsuoka}]
We proceed by noetherian induction on $S/I$ and induction on $d=\dim S$. The case $d=1$ is immediate: $\dim(S/I)=1=\dim S$, so $I=0$.
 
Assume that $d\ge 2$. By the assumption, there exists an element $g\in \nn^2$ such that $\nn^2=I+(g)$.

\textsf{Step 1:} We note that $g\notin \nn^3$, otherwise by Nakayama's lemma $\nn^2=I$, contracting the fact that $\dim(S/I)=1$. Note that $g\notin \sqrt{I}$, for the same reason. 

\textsf{Step 2:} Since $d\ge 2$, by \Cref{prop_1dimCM_gstretched}, $S/I$ is not Cohen--Macaulay, namely $\depth(S/I)=0$. Hence $\nn \in \Ass(S/I)$. Let $I=Q_1\cap \cdots \cap Q_s \cap L$ be a minimal primary decomposition of $I$, where $\sqrt{L}=\nn$. Note that $s\ge 1$ as $\dim(S/I)=1$. Each $Q_i$ is a primary ideal of height $d-1$. First we show that $s=1$.

Assume that $s\ge 2$. Fix $1\le i\le s$. Since $s\ge 2$,  $J:=Q_i\cap \nn^2$ strictly contains $I$. Clearly $\nn^2=I+(g)$ implies $\nn^2=J+(g)$, and $\dim(S/J)=1$. Hence by the hypothesis of the noetherian induction, $J=U\nn$ for some $U$ generated by a regular sequence of $d-1$ elements in $\nn\setminus \nn^2$. Thanks to  Equation \eqref{eq_intersect_Ipowersofm} in \Cref{ex_regularlocal}, we have $J=U\nn=U\cap\nn^2$ which is a minimal primary decomposition of $J$. This implies that $Q_i \cap \nn^2=U\cap \nn^2$, hence $Q_i=U$.

But then $\nn^2=I+(g)\subseteq Q_1\cap Q_2+(g)$, and this is a contradiction by Lemma \ref{lem_forbidden_containment}. Hence $s=1$. 

\textsf{Step 3:} Denote $Q=Q_1$, $J=\sqrt{Q}\cap \nn^2$. We claim that $I=J$. Assume the contrary, that  $I\subsetneq J$. Clearly $I+(g) = \nn^2 \subseteq J+(g)\subseteq \nn^2$, so equalities happen throughout. Moreover, $\dim(S/J)=1$, we conclude by the induction hypothesis that $J=U\nn$ for some $U$ generated by a regular sequence of $d-1$ elements in $\nn\setminus \nn^2$. This implies that $\sqrt{Q}\cap \nn^2=J=U\cap \nn^2$, hence $\sqrt{Q}=U$. Also $I\subseteq J=\nn U$.

Assume that $x_1,\ldots,x_d$ minimally generate $\nn$ and $U=(x_1,\ldots,x_{d-1})$. Since $\nn^2\subseteq U+(g)$, we get $g\notin U$. Now $x_d^2\in U\nn+(g)$, so $x_d^2=\alpha g-u$, for some $\alpha\in S, u\in U\nn$. Note that $u=\alpha g-x_d^2 \in U\nn=\nn^2\cap U$.

If $\alpha \in \nn$, then $x_d^2=\alpha g-u\in \nn^3+\nn U$. Now
\[
\nn^2=(U+(x_d))^2=\nn U+(x_d^2) \subseteq \nn U+\nn^3,
\]
hence by Nakayama's lemma $\nn^2=\nn U$. This is a contradiction.

Thus $\alpha \notin \nn$. So we can assume $\alpha=1$, namely $g=x_d^2+u$.

For $1\le i\le d-1, 1\le j\le d$, since $x_ix_j\in \nn^2=I+(g)$, we deduce $x_ix_j=u_{ij}+c_{ij}g$, where $u_{ij}\in I\subseteq U, c_{ij}\in S$. As $g\notin U=(x_1,\ldots,x_{d-1})$ and $c_{ij}g=x_ix_j-u_{ij}\in U$, we get $c_{ij}\in U$. Thus $x_ix_j\in I+Ug \subseteq I+\nn^2 U$ for any $1\le i\le d-1,1\le j\le d$, and hence $U\nn\subseteq I+\nn^2U$. By Nakayama's lemma, this implies $\nn U \subseteq I \subseteq \nn U$. This is a contradiction to the assumption $I\neq \sqrt{Q}\cap \nn^2=U\cap \nn^2=\nn U$. Hence $I=J=\sqrt{Q}\cap \nn^2$.

\textsf{Step 4:} By Step 3, we can assume that $Q$ is a prime ideal strictly contained in $\nn$, and $I=Q\cap \nn^2$. If $Q\subseteq \nn^2$ then $I=Q$ and $\depth(S/I)>0$, a contradiction. Hence $Q$ contains an element, say $z$, in $\nn\setminus\nn^2$.

Since $\dim(S/(I+(z)))=1$ (as $Q$ is the only minimal prime of $I+(z)$), working in $S/(z)$ and using the induction on $\dim(S)$, we get $I+(z)=\nn U+(z)$ for some ideal $U\supseteq (z)$ of $S$ such that $U/(z)$ is generated by a regular sequence of $d-2$ elements in $\nn\setminus \nn^2$ in $S/(z)$, say $U/(z)=(z_2+(z),\ldots,z_{d-1}+(z))$. This means that $U=(z,z_2,\ldots,z_{d-1})$ is generated by a regular sequence of $d-1$ elements in $\nn \setminus \nn^2$. In particular, $U$ is a prime ideal of $S$.

Now $I+(z)=Q\cap \nn^2+(z)$, so it is easy to see that $\{Q\}=\Min(I+(z))=\Min(\nn U+(z))=\{U\}$. Hence $Q=U=(z,z_2,\ldots,z_{d-1})$. This finishes the inductions and the proof. 
\end{proof}
\begin{cor}
\label{cor_charofgstr} 
Let $(R,\mm)$ be a complete noetherian local ring. Then $R$ is  a one-dimensional $g$-stretched  ring if and only if $R$ has the form $R=S/Q\nn$ where $S$ is a complete regular local ring of dimension $d$ and $Q$ is generated by a regular sequence of $d-1$ elements in $\nn\setminus \nn^2$.
\end{cor}
\begin{proof}
By Cohen's structure theorem \cite[Theorem 29.4]{Mat}, $R$ can be expressed as the homomorphic image $S/I$ of a complete regular local ring  $(S,\nn)$ with $I\subseteq \nn^2$. Then the assertion immediately follows from Theorem \ref*{thm_Matsuoka}.
\end{proof} 

\section{Characterization of $g$-stretched Koszul rings}
\label{sect_charact_Koszulness}

The main result of this section strengthens a result due to Avramov, Iyengar and \c{S}ega \cite[Theorem 4.1]{AIS}. These authors only consider rings $(R,\mm)$ such that $R$ is $g$-stretched  and $\mm^3=0$, while we consider all $g$-stretched local rings. 
\begin{thm}
\label{thm_Koszul}
Let $R$ be a $g$-stretched local ring. Then the following are equivalent: 
\begin{enumerate}[\quad \rm (1)]
 \item $R$ is Koszul;
  \item Either $\dim R=1$, or $R$ is artinian, $\mm^3=0$, and $r(R)\le \codim(R)-1$ unless $\mm^2=0$.
\end{enumerate}
\end{thm}
The following remark will allow us to reduce to the complete case; note that if $R$ is artinian then $r(R)=\rk_k(0:_R \mm)$ and $\codim(R)=\mu(\mm)$.
\begin{rem}
\label{rem_complete}
Let $(R,\mm)$ be a noetherian local ring, and $(\wh{R},\wh{\mm})$ its $\mm$-adic completion. Then the following statements hold.
\begin{enumerate}[\quad \rm (1)]
 \item $R$ is Koszul if and only if $\wh{R}$ is so.
 \item For each $i\ge 1$, there is an equality $\mu(\mm^i)=\mu(\wh{\mm}^i)$. In particular, $R$ is $g$-stretched if and only if $\wh{R}$ is so, and $\mm^i=0$ if and only if $\wh{\mm}^i=0$ for each $i\ge 1$.
 \item $\dim R=\dim \wh{R}$.
 \item $\rk_k(0:_R\mm)=\rk_k(0:_{\wh{R}}\wh{\mm})$.
\end{enumerate}
Indeed, (1) was mentioned in \Cref{prop.passtocompletion}, and it is also a consequence of the fact that for all $i\ge 0$, $\wh{\mm}^i/\wh{\mm}^{i+1} \cong \mm^i/\mm^{i+1}$, and so $\gr_\mm(R)\cong \gr_{\wh{\mm}}(\wh{R})$. The same fact implies (2), while (3) is classical.

For (4), let $I=(0:_R\mm)$. By flatness, $(0:_{\wh{R}}\wh{\mm})=I\wh{R}=\wh{I}$. Using $\mm I=0$, we get
\[
I\cong I\ot_R \frac{R}{\mm} \cong (I\ot_R \wh{R})\ot_{\wh{R}}\frac{\wh{R}}{\wh{\mm}} \cong \wh{I} \ot_{\wh{R}}\frac{\wh{R}}{\wh{\mm}} \cong \wh{I}.
\]
Hence the natural map $I\to \wh{I}$ is an isomorphism, and $\rk_k(0:_R\mm)=\mu(I)=\mu(\wh{I})=\rk_k(0:_{\wh{R}}\wh{\mm})$, as desired.
\end{rem}

One ingredient for the implication (1) $\Longrightarrow$ (2) in \Cref{thm_Koszul} is
\begin{lem}
\label{lem_artinKoszul}
Let $(R,\mm,k)$ be a $g$-stretched local ring, that is artinian and Koszul. Then $\mm^3=0$.
\end{lem}
\begin{proof}
We first reduce to the case when $(R,\mm,k)$ is a standard graded $k$-algebra.

Note that since $R$ is artinian and Koszul, so is the associated graded ring $\gr_\mm(R)$. Assume that $\mm^2$ is the principal ideal $(a)$. Passing to $\gr_\mm(R)$, we also see that the graded maximal ideal $\mm^g=\bigoplus_{i\ge 1} \mm^i/\mm^{i+1}$ has the property that $(\mm^g)^2$ is the principal ideal $(a+\mm^3)$. Hence we reduce to the following statement:

\emph{Let $(R,\mm,k)$ be a standard graded $k$-algebra. Assume that $R$ is artinian, Koszul, and $\mm^2$ is a principal ideal. Then $\mm^3=0$.}

Indeed, applying this for $\gr_\mm(R)$, we get $(\mm^g)^3=0$ namely $\mm^3=\mm^4$. Nakayama's lemma then takes care of the rest.

For the statement, we use induction on the minimal number of homogeneous generators of $\mm$. If $\mm=0$ then there is nothing to do. Assume that $\mu(\mm)\ge 1$. We can also assume that $\mm^2\neq 0$.

Since $R$ is artinian, for some $s\ge 3$ we have $\mm^s=0\neq \mm^{s-1}$. Let $\mm^2=(a)$, where $a$ is homogeneous of degree 2. Consider the exact sequence
\[
0\to (0:a)(-2) \to R(-2) \xrightarrow{\cdot a} \mm^2 \to 0.
\]
Since $R$ is Koszul, $\reg_R \mm^2=2$. The above exact sequence yields a chain
\[
\reg_R (0:a) \le \max\{\reg_R R+2, \reg_R \mm^2+1\}-2=1.
\]
Since $0\neq \mm^{s-1} \subseteq (0:a)$, we conclude that $\reg_R (0:a)=1$. This implies that $(0:a)$ is generated by linear forms, say $(0:a)=(x_1,\ldots,x_d)$ where $d\ge 1$.

Consider the $k$-algebra $R/(x_1,\ldots,x_d)$. Since $\reg_R R/(x_1,\ldots,x_d)=0$, and $R$ is Koszul, by \cite[Section 3.2, Theorem 2]{CDR}, $\overline{R}=R/(x_1,\ldots,x_d)$ is also Koszul. Let $\overline{\mm}=\mm/(x_1,\ldots,x_d)$. By the induction hypothesis, $\overline{\mm}^3=0$, namely $\mm^3\subseteq (x_1,\ldots,x_d)$.

Take any homogeneous element $u\in \mm^3$. Since $u$ has degree at least 3, and $x_1,\ldots,x_d$ are linear forms, 
\[
u=a_1x_1+\cdots+a_dx_d, a_i\in \mm^2.
\]
But $\mm^2=(a)$ and $(0:a)=(x_1,\ldots,x_d)$, so $a_ix_i=0$. This implies that $u=0$, namely $\mm^3=0$. This concludes the proof.
\end{proof}
\begin{proof}[Proof of \Cref{thm_Koszul}]
By \Cref{rem_complete}, we may assume that $R$ is complete.

(1) $\Longrightarrow$ (2). Suppose $R$ is Koszul. Since $R$ is a $g$-stretched ring, $\dim R\le 1$. If $\dim R<1$ then $R$ is artinian. By \cite[Theorem 4.1]{AIS} it is sufficient to show that $\mm^3=0$. This follows from \Cref{lem_artinKoszul}.

(2) $\Longrightarrow$ (1). Suppose that  $R$ is artinian, $\mm^3=0$, and $\rk_k(0:\mm)\le \mu(\mm)-1$. Then by \cite[Theorem 4.1]{AIS} $R$ is Koszul. For the rest of the proof we consider the following two cases.

If $\dim R=1$ then as $R$ is complete, by \Cref{cor_charofgstr} there is a complete regular local ring $(S,\nn)$ of dimension $d$ such that  $R=S/(x_1,x_2,\ldots,x_{d-1})\nn$, where $x_1,x_2,\ldots,x_{d-1}$ is a regular sequence in $\nn\setminus \nn^2$. We may assume $\nn=(x_1,x_2,\ldots,x_{d-1},x_d)$. Let $I=(x_1,x_2,\ldots,x_{d-1})\nn\subseteq S$ and $\mm=\nn/I$. We consider the set 
$$\mathcal{F}=\{(0),(\ovl{x_1}),(\ovl{x_1},\ovl{x_2}),\ldots,(\ovl{x_1},...,\ovl{x_{d-1}}),\mm\}$$
of ideals of $R$, where $\ovl{x_i}$ denotes the image of $x_i$ in $R$ for all $i$. Then the following \Cref{lem_Koszulfiltr} shows that $\mathcal{F}$ is a Koszul filtration of $R$. Hence $R$ is Koszul, and the proof of the implication (2) $\Longrightarrow$ (1) is completed.
\end{proof}
\begin{lem}
\label{lem_Koszulfiltr}
With the notation as in the proof of \Cref{thm_Koszul}, $\mathcal{F}$ is a Koszul filtration of $R$.
\end{lem}
\begin{proof} The first condition in \Cref{dfnKoszulfil} is automatically satisfied for $\mathcal{F}$. For the second condition we need to show for all $1\le i\le d$ and $s\ge 1$ that
\begin{align}\label{equal1}
(\ovl{x_1},\ovl{x_2},\ldots,\ovl{x_i})\cap \mm^{s+1}=(\ovl{x_1},\ovl{x_2},\ldots,\ovl{x_i})\mm^{s}.
\end{align}
It is clear that the equality (\ref{equal1}) holds true for $i=d$, because $\mm=(\ovl{x_1},\ovl{x_2},\ldots,\ovl{x_{d}})$. Assume that $i\le d-1$. Note that since $(\ovl{x_1},\ldots,\ovl{x_{d-1}})\mm=0$, 
\[
\mm^{s+1}=(\ovl{x_d}^{s+1}) \quad \text{for all $s\ge 1$}.
\]
We have to show that $(\ovl{x_1},\ovl{x_2},\ldots,\ovl{x_i})\cap \mm^{s+1} =0$ for $i\le d-1$.

Now $(\ovl{x_1},\ovl{x_2},\ldots,\ovl{x_i})\cap \mm^{s+1}=(\ovl{x_1},\ovl{x_2},\ldots,\ovl{x_i})\cap (\ovl{x_d}^{s+1})$. Take $f\in S$ such that $\ovl{f}\ovl{x_d}^{s+1} \in (\ovl{x_1},\ovl{x_2},\ldots,\ovl{x_i})$ then 
$$
fx_d^{s+1}\in (x_1,x_2,\ldots,x_i)+I\subseteq (x_1,x_2,\ldots,x_{d-1}).
$$
This yields  $f\in (x_1,x_2,\ldots,x_{d-1})$, proving the equality in \eqref{equal1}. 

\quad For the last condition in \Cref{dfnKoszulfil}, it suffices to observe that
\begin{align*}
 (\ovl{x_1},\ldots,\ovl{x_{i-1}}): \ovl{x_i} &= \mm, \quad \text{for $1\le i\le d-1$},\\
  (\ovl{x_1},\ldots,\ovl{x_{d-1}}): \ovl{x_d} &= (\ovl{x_1},\ldots,\ovl{x_{d-1}}).
\end{align*}
This concludes the proof.
\end{proof}
\begin{ex}
	Let $R=k[[x,y]]/(x^3,xy,y^2)$. Then $\mm^2=(x^2)\neq 0$ and $\mm^3=0$. Hence $R$ is $g$-stretched, artinian, $\mm^3=0$, but $R$ is not Koszul as $\gr_\mm(R)$ is not even quadratic. In this case $(0:\mm)=(x^2,y)$ has rank $2>\mu(\mm)-1=1$. This shows that the condition $\rk_k(0:\mm)\le \mu(\mm)-1$ unless $\mm^2= 0$ in \Cref{thm_Koszul} is necessary.
\end{ex}

We have already known from Theorem \ref{thm_Koszul} that every one-dimensional $g$-stretched ring  is Koszul. The following shows that such rings are also absolutely Koszul.

\begin{prop}
\label{prop_dim1_absKos}
Let $(R,\mm)$ be a $g$-stretched  local ring of dimension 1.
\begin{enumerate}[\quad \rm (1)]
 \item  If $\depth R=1$, then $R$ is a DVR, and $\glind R=\dim R=1$.
 \item  If $\depth R=0$, then $R$ is both Koszul and Golod; in particular, $\glind R\le 2\embdim R$.
\end{enumerate}
In either case, $R$ is absolutely Koszul of finite global linearity defect.
\end{prop}
\begin{proof}
(1) If $\depth R=1$, we know from Proposition \ref{prop_1dimCM_gstretched} that $R$ is a DVR. Thus  $\glind R=\dim R=1$.

(2) Passing to the $\mm$-adic completion does not change the statement. Indeed, since $\depth \wht{R}=\depth R=0$ and $\embdim \wht{R}=\embdim R$, $R$ is Koszul (or Golod) if and only if $\wht{R}$ has the corresponding property, thanks to \Cref{rem_complete} and \Cref{lem.Golod.passgetocompletion}. By \Cref{prop.passtocompletion}, $\glind \wht{R}=\glind R$ and $R$ is absolutely Koszul if and only if $\wht{R}$ is so. Thus we may assume that $R$ is a complete noetherian local ring. In particular, we may write $R=S/I$, where $(S,\nn)$ is a complete regular local ring, and $I\subseteq \nn^2$.

If $\depth R=0$, denote $d=\dim S \ge \dim R=1$. By Theorem \ref{thm_Matsuoka}, $I=Q\nn$ for some ideal $Q$ generated by a regular sequence of $(d-1)$ elements in $\nn\setminus \nn^2$. The proof of Theorem \ref{thm_Koszul} implies that $R$ is Koszul. To show that $R$ is Golod, we use a result of Ahangari Maleki \cite[Lemma 2.1]{AM19}. As $Q^2\subseteq I= Q\nn \subseteq Q$, the last result implies that $S/I=S/(Q\nn)$ is Golod if we can show that the map $\Tor^S_i(S/(Q\nn),k)\to \Tor^S_i(S/Q,k)$ induced by the natural projection $S/(Q\nn) \to S/Q$, is zero for all $i\ge 1$.

Now $Q$ is a Koszul $S$-module by Example \ref{ex_regularlocal}. Hence by \Cref{thm_small_inclusion}, the natural map $\Tor^S_i(Q\nn,k)\to \Tor^S_i(Q,k)$ induced by the injection $Q\nn \to Q$, is zero for all $i\ge 0$. This is the desired conclusion.

Hence $R$ is both Koszul and Golod, so by a result of Herzog--Iyengar \cite[Corollary 6.2]{HIy}, $\glind R\le 2\embdim R$. The proof is concluded.
\end{proof}
On the absolute Koszulness of artinian $g$-stretched rings we have the following assertion.
\begin{prop}
\label{prop_dim0_absKos}
Let $(R,\mm)$ be an artinian $g$-stretched ring. If $R$ is Koszul, then it is absolutely Koszul.
\end{prop}
\begin{proof}
By \Cref{thm_Koszul}, we get $\mm^3=0$, and $\rk_k(0:\mm)\le \mu(\mm)-1$ unless $\mm^2=0$. If $\mm^2=0$, then for any finitely generated $R$-module $M$, its first syzygy module is killed by $\mm$, hence is a direct sum of copies of $k$. As $k$ is a Koszul module, $\lind_R M\le 1$ for any $R$-module $M$. 

Assume that $\mm^2\neq 0$, so $\dim_k \mm^2=1$ as $R$ is $g$-stretched. By \cite[Theorem 1.1 and Theorem 4.1, (i) $\Longleftrightarrow$ (v)]{AIS}, every finitely generated $R$-module $M$ has a Koszul syzygy module, namely $\lind_R M$ is finite. Hence $R$ is absolutely Koszul.
\end{proof}

\section{\Cref{thm_main_HerzogIyengar} in the case $\dim R=0$}
\label{sect_dim0_case}

The main result of this section is that \Cref{quest_HIy} has a positive answer for $g$-stretched local rings whose residue fields have characteristic zero. In particular, this confirms the case $R$ is artinian in \Cref{thm_main_HerzogIyengar}: as mentioned in the introductory \Cref{sect_intro}, $R$ is an artinian ring of almost minimal multiplicity if and only if $R$ is an artinian $g$-stretched ring with $\mm^3=0\neq \mm^2$.
\begin{thm}
\label{thm_HIy_gstretched_long}
Let $(R,\mm,k)$ be a noetherian local ring with $\chara(k)=0$. Assume that $R$ is $g$-stretched. Then the following are equivalent:
\begin{enumerate}[\quad \rm (1)]
\item $R$ is absolutely Koszul;
\item $R$ is Koszul;
\item $\lind_R k<\infty$;
\item Either $\dim R=1$, or $R$ is artinian, $\mm^3=0$, and $r(R)\le \codim(R)-1$ unless $\mm^2=0$.
\end{enumerate}
\end{thm}
Note that, by Theorem \ref{thm_Koszul}, (2) $\Longleftrightarrow$ (4). The critical implication in \Cref{thm_HIy_gstretched_long} is (3) $\Longrightarrow$ (2).  If $R$ is $g$-stretched and $\dim R=1$  then $R$ is Koszul by Theorem \ref{thm_Koszul}. We therefore just need to treat the artinian case of (3) $\Longrightarrow$ (2). The harder situation is when $R$ is  artinian and Gorenstein. We treat the easier non-Gorenstein case first.
\begin{prop}
\label{prop_artin_nonGor_stretched}
Let $(R,\mm,k)$ be a $g$-stretched artinian ring. Assume furthermore that $\chara(k)=0$. If $\lind_R k<\infty$  and $R$ is not Gorenstein then $R$ is a Koszul ring.
\end{prop}

A key ingredient is the following structural result on ideals defining stretched local rings.
\begin{prop}[Elias-Valla {\cite[Theorem 3.1]{ElV08}}]
\label{prop_EliasValla}
Let $(S,\nn,k)$ be a regular local ring, $I\subseteq \nn^2$ an ideal such that $R=S/I$ is artinian and $g$-stretched. Assume that $\chara(k)=0$. Let $\mm:=\nn/I, h=\mu(\mm)$, and $\tau=r(R)$. Let $s$ be the socle degree of $R$.
\begin{enumerate}[\quad \rm (1)]
 \item If $\tau < h$, then we can find a minimal generating set $y_1,\ldots,y_h$ of $\nn$ such that $I$ is minimally generated by the elements $\{y_iy_j\}_{1\le i< j\le h}$, $\{y_j^2\}_{2\le j\le \tau}$, and $\{y_i^2-u_iy_1^s\}_{\tau+1\le i\le h}$. Here $u_i$ are units in $S$.
 \item If $\tau = h$, then we can find a minimal generating set $y_1,\ldots,y_h$ of $\nn$ such that $I$ is minimally generated by the elements $\{y_1y_j\}_{2\le  j\le h}$, $\{y_iy_j\}_{2\le i\le j\le h}$, and $y_1^{s+1}$. 
\end{enumerate}
In either case, we have $\mm^i=(y_1^i)$ for all $i\ge 2$.
\end{prop}

\begin{proof}[Proof of Proposition \ref{prop_artin_nonGor_stretched}]
By \Cref{prop.passtocompletion} and \Cref{rem_complete}, we may assume that $R$ is complete. By Cohen's structure theorem, we may write $R=S/I$, where $(S,\nn,k)$ is a complete regular local ring, $I\subseteq \nn^2$.  Let $\tau$ be the type of $R$, $h=\mu(\mm)$ and $s$ the socle degree of $R$. Note that $\tau\ge 2$ since $R$ is not Gorenstein. By Proposition \ref{prop_EliasValla}, there are two cases.
 
 \textsf{Case 1:} $\tau=h$. We can find a minimal generating set $y_1,\ldots,y_h$ of $\nn$ such that $I$ is minimally generated by the elements $\{y_1y_j\}_{2\le  j\le h}$, $\{y_iy_j\}_{2\le i\le j\le h}$, and $y_1^{s+1}$. We claim the followings:
 \begin{enumerate}[\quad \rm (i)]
 \item $\mm=(y_2)\oplus (y_1,y_3,y_4,\ldots,y_h)$,
  \item $(0):y_2=\mm$.
 \end{enumerate}
(i): Take $a\in (y_2)\cap (y_1,y_3,\ldots,y_h)$. Assume that $a\neq 0$. Since $\mm (y_2)=0$, one can assume that $a=a_2y_2$ where $a_2\notin \mm$. But then $y_2$ is not a minimal generator of $\mm$, a contradiction. Hence $\mm=(y_2)\oplus (y_1,y_3,\ldots,y_h)$. 
 
 (ii) This holds since $y_2\mm=0$ and $y_2\neq 0$.
 
 Now we show that if $\lind_R k<\infty$, then $R$ is Koszul. Assume the contrary, that $R$ is not Koszul. Then $\lind_R k \ge 1$ and $\lind_R \mm=\lind_R k-1$ by \Cref{prop_ld_exactseq}(2). From the exact sequence
 \[
0\to k =R/((0):y_2) \to R \to R/(y_2) \to 0
 \]
and \Cref{prop_ld_exactseq}(2), we conclude that $\lind_R R/(y_2)=\lind_R k+1=\lind_R \mm+2$, and so $\lind_R (y_2)=\lind_R R/(y_2)-1=\lind_R \mm+1$, again by \Cref{prop_ld_exactseq}(2). But by (i), $(y_2)$ is a direct summand of $\mm$, so $\lind_R \mm \ge \lind_R (y_2)$. This implies that $\lind_R \mm=\lind_R (y_2)=\infty$, whence $\lind_R k=\infty$ which is a contradiction.

\textsf{Case 2}: $\tau<h$. By Proposition \ref{prop_EliasValla}, we can find a minimal generating set $y_1,\ldots,y_h$ of $\nn$ such that $I$ is minimally generated by the elements $\{y_iy_j\}_{1\le i< j\le h}$, $\{y_j^2\}_{2\le j\le \tau}$, and $\{y_i^2-u_iy_1^s\}_{\tau+1\le i\le h}$, where $u_i$ are units in $S$. Similarly as in Case 1, we have
\begin{enumerate}[\quad \rm (iii)]
 \item $\mm=(y_2)\oplus (y_1,y_3,y_4,\ldots,y_h)$,
  \item[(iv)] $(0):y_2=\mm$.
 \end{enumerate}
 Moreover, we deduce that if $\lind_R k<\infty$, then it is Koszul, in the same manner. This concludes the proof.
\end{proof}
The main work in the proof of \Cref{thm_HIy_gstretched_long} is done by
\begin{thm}
\label{thm_HIy_artinGor}
Let $(R,\mm,k)$ be a $g$-stretched artinian local ring.  Assume that $\chara(k)=0$. If $\lind_R k$ is finite and $R$ is Gorenstein, then $R$ is a Koszul ring.
\end{thm}
In order to prove Theorem \ref{thm_HIy_artinGor} we need to use the discussion of weak Koszul filtration.
\begin{proof}[Proof of \Cref{thm_HIy_artinGor}]
Again by \Cref{prop.passtocompletion} and \Cref{rem_complete}, we may assume that $R=S/I$, where $(S,\nn,k)$ is a complete regular local ring, $I\subseteq \nn^2$. Let $\tau$ be the type of $R$, $h=\mu(\mm)$ and $s$ the socle degree of $R$. Note that $\tau=1$ since $R$ is Gorenstein. By abuse of notations, we employ the same notations for an element and its image in the relevant quotient rings.

If $h=1$, then $I=(y_1^{s+1})$ where $y_1$ is a minimal generator of $\nn$. Assume that $R$ is not Koszul, then \Cref{rem_ldge1} implies that $\lind_R \mm=\lind_R (y_1)\ge 2$. From the following equalities in $R$: $(0):y_1  =(y_1^s),$ and $(0):y_1^s = (y_1)$, and \Cref{thm_small_inclusion}, we deduce that $\lind_R (y_1^s) \ge \lind_R (y_1)+1$ and $\lind_R (y_1) \ge \lind_R (y_1^s)+1$. Hence $\lind_R (y_1^s) = \lind_R (y_1)=\infty$. It yields that $\lind_R k=\lind_R\mm+1 =\lind_R(y_1)+1=\infty$ which is impossible. Hence $R$ is Koszul.

Assume that $h\ge 2$. By Proposition \ref{prop_EliasValla}, we can find a minimal generating set $y_1,\ldots,y_h$ of $\nn$ such that $I$ is minimally generated by the elements $\{y_iy_j\}_{1\le i< j\le h}$, and $\{y_i^2-u_iy_1^s\}_{2\le i\le h}$, where $u_i$ are units in $S$.

\textsf{Step 1:} We treat the case $h=2$ first. In this case $\mm=(y_1,y_2)$ and $I=(y_1y_2,y_2^2-u_2y_1^s)$.

We have the following equalities in $R$:
\begin{enumerate}[\quad \rm (a)]
 \item $(y_1):y_2=\mm$.
 \item $(0):y_1=(y_2)$.
 \item $(0):y_2=(y_1)$.
\end{enumerate}
The proofs of (a), (b), (c) are similar to that of (i), (iii), (iv) in Step 2 below, resp., hence we leave the details to the interested reader. In particular, $R$ has a weak Koszul filtration $\Fc=\{(0), \mm, (y_1), (y_2)\}$. The above equalities yield the exact sequences
\begin{displaymath}
 \xymatrix{
0 \ar[r] & \dfrac{R}{\mm} \ar[r]^{\cdot y_2}&\dfrac{R}{(y_1)}  \ar[r] & \dfrac{R}{\mm} \ar[r] &  0 \\
0 \ar[r] & \dfrac{R}{(y_2)} \ar[r]^{\cdot y_1}& R  \ar[r] & \dfrac{R}{(y_1)} \ar[r] &  0 \\
0 \ar[r] & \dfrac{R}{(y_1)} \ar[r]^{\cdot y_2}& R  \ar[r] & \dfrac{R}{(y_2)} \ar[r] &  0
 }
\end{displaymath}
Thanks to \Cref{lem_weakfiltr_Tormap}, each of these exact sequence has the form $0\to M \xlongrightarrow{\phi} P \to N \to 0$, where $\phi$ is Tor-vanishing.  Assume that $\lind_R k<\infty$ but $R$ is not Koszul. Then $\lind_R k\ge 1$.
By \Cref{prop_ld_exactseq}(3) and the first exact sequence,
\begin{align*}
\lind_R k &\le \max\{\lind_R R/(y_1),\lind_R k-1\},\\
\lind_R R/(y_1)  & \le \lind_R k.
\end{align*}
Hence $\lind_R R/(y_1)  = \lind_R k\ge 1$. From the remaining exact sequences and \Cref{prop_ld_exactseq}(2), we get $\lind_R R/(y_2)=\lind_R R/(y_1)-1$ and $\lind_R R/(y_2)=\lind_R R/(y_1)+1$. This is a contradiction as $\lind_R R/(y_1)=\lind_R k<\infty$. Hence if $h=2$, $R$ is Koszul.

\textsf{Step 2:} Next, consider the case $h\ge 3$. Consider the collection $\Fc$ of ideals of $R$ consisting of the following ideals
\begin{gather*}
(0), \mm, (y_1,\ldots,y_i), 1\le i\le h-1, \quad \text{and} \quad (y_j,\ldots,y_h), 2\le j\le h.
\end{gather*}
Note that the ideals of $\Fc$ are generated by subsets of $y_1,\ldots,y_h$, so the conditions (WF1), (WF2) in \Cref{dfn_wKf} are fulfilled. We claim the followings, which imply that $\Fc$ is a weak Koszul filtration for $R$:
\begin{enumerate}[\quad \rm (i)]
 \item $(y_1):y_i=(y_1,y_2,\ldots,y_{i-1}):y_i=\mm$ for every $2\le i\le h$;
  \item $(y_h):y_j=(y_{j+1},y_{j+2},\ldots,y_h):y_j=\mm$ for every $2\le j\le h-1$;
 \item $(0):y_1=(y_2,\ldots,y_h)$;
 \item $(0):y_h=(y_1,y_2,\ldots,y_{h-1})$.
\end{enumerate}

For (i): Since $(y_1):y_i \subseteq (y_1,y_2,\ldots,y_{i-1}):y_i \subseteq \mm$, it suffices to show that $\mm \subseteq (y_1):y_i$. Clearly $y_j \in (y_1):y_i$ for $j\in [h]\setminus \{i\}$, as $y_iy_j=0$. Note that $y_i^2=u_iy_1^s\in (y_1)$, so $y_i \in (y_1):y_i$. Thus $\mm \subseteq (y_1):y_i$, as claimed.

For (ii): Since $(y_h):y_j \subseteq (y_{j+1},y_{j+2},\ldots,y_h):y_j\subseteq \mm$, it suffices to show that $\mm \subseteq (y_h):y_j$. Arguing similarly as for (i), noting that $u_hy_j^2-u_jy_h^2=0$, so $u_hy_j \in (y_h):y_j$, which yields $y_j \in (y_h):y_j$ as $u_h$ is a unit.

For (iii): Clearly $(0):y_1 \supseteq (y_2,\ldots,y_h)$. Take $a\in (0):y_1$. Since $a\in \mm$, we can write $a=y_1a_1+b_1$, $b_1\in (y_2,\ldots,y_h)$. Then $ay_1=a_1y_1^2=0$. If $2<s+1$ then this yields $a_1\in \mm$, which in turn  implies that $a=y_1^2a_2+b_2$, where $a_2\in R, b_2 \in (y_2,\ldots,y_h)$. Continuing this argument, we conclude that $a\in y_1^sa_s+b_s$, $a_s\in R, b_s\in (y_2,\ldots,y_h)$. Now $u_2a\in (u_2y_1^s)+(u_2b_s) \subseteq (y_2,\ldots,y_h)$ and $u_2$ is a unit, so $a\in (y_2,\ldots,y_h)$, as desired.

For (iv): The proof is similar to (iii), noting that $y_h^2=u_hy_1^s\neq 0$,  $y_h^3=u_hy_1^sy_h=0$, and that $u_2y_h^2=u_hy_2^2$.

\textsf{Step 3:} We have seen that for each ideal $(0)\neq I\in \Fc$, there exists an ideal $J\in \Fc$ such that $I=J+(x)$ for some $x\in R$ and $J:x\in \Fc$. In particular, we have an exact sequence
\[
0\to \dfrac{R}{J:x} \xrightarrow{\cdot x} \dfrac{R}{J} \to \dfrac{R}{I} \to 0.
\]
Since $\Fc$ is a weak Koszul filtration for $R$, by Lemma \ref{lem_weakfiltr_Tormap}, the injective map in the last exact sequence is Tor-vanishing.

\textsf{Step 4:} From Step 2 and Step 3, we have the various exact sequences of $R$-modules of the form $0\to M \xrightarrow{\phi} P \to N \to 0$ where $\phi$ is Tor-vanishing. In details, from (i), we have the exact sequences
\begin{displaymath}
\xymatrixcolsep{6.72mm}
\xymatrixrowsep{4.8mm}
\xymatrix{
0  \ar[r] &  \dfrac{R}{\mm} \ar[rr]^{\cdot y_h} && \dfrac{R}{(y_1,\ldots,y_{h-1})} \ar[r] & \dfrac{R}{\mm} \ar[r] & 0\\ 
0  \ar[r] & \dfrac{R}{\mm} \ar[rr]^{\cdot y_{h-1}}&& \dfrac{R}{(y_1,\ldots,y_{h-2})} \ar[r] & \dfrac{R}{(y_1,\ldots,y_{h-1})}  \ar[r] &  0\\
& \cdots  && \cdots  & \cdots & \\
0 \ar[r] & \dfrac{R}{\mm} \ar[rr]^{\cdot y_2}&& \dfrac{R}{(y_1)}  \ar[r] & \dfrac{R}{(y_1,y_2)} \ar[r] &  0.
}
\end{displaymath}
Now since $\Tor^R_i(k,\cdot y_h)=0$ for all $i$, we apply Proposition \ref{prop_ld_exactseq}(3) to the first exact sequence of the above display. This yields
\begin{align*}
\lind_R(R/\mm) &\le \max\{\lind_R R/(y_1,\ldots,y_{h-1}),\lind_R(R/\mm)-1\},\\
\lind_R R/(y_1,\ldots,y_{h-1}) & \le \lind_R (R/\mm).
\end{align*}
Hence $\lind_R R/(y_1,\ldots,y_{h-1}) = \lind_R (R/\mm)$. Arguing similarly, from the second exact sequence we get $\lind_R R/(y_1,\ldots,y_{h-2}) = \lind_R R/(y_1,\ldots,y_{h-1}) = \lind_R (R/\mm)$.
Continuing this way, we get 
\begin{equation}
\label{eq_lind_1}
\lind_R R/(y_1)=\lind_R R/(y_1,y_2)=\cdots=  \lind_R R/(y_1,\ldots,y_{h-1}) = \lind_R (R/\mm).
\end{equation}
Assume that $R$ is not Koszul, then all the numbers in the last display are positive.

From (iii) in Step 2, we get the exact sequence
\[
0 \longrightarrow \dfrac{R}{(y_2,\ldots,y_h)} \xrightarrow{\cdot y_1} R \longrightarrow \dfrac{R}{(y_1)} \longrightarrow 0.
\]
Thanks to Proposition \ref{prop_ld_exactseq}(2) and the fact that $\lind_R R/(y_1)\ge 1$, we get
\[
\lind_R R/(y_2,\ldots,y_h)=\lind_R R/(y_1) -1 =\lind_R (R/\mm)-1.
\]
From (ii) in Step 2, we get the exact sequences
\begin{displaymath}
\xymatrixcolsep{6.72mm}
\xymatrixrowsep{4.8mm}
\xymatrix{
0  \ar[r] &  \dfrac{R}{\mm} \ar[rr]^{\cdot y_2} && \dfrac{R}{(y_3,\ldots,y_h)} \ar[r] & \dfrac{R}{(y_2,\ldots,y_h)} \ar[r] & 0\\ 
0  \ar[r] & \dfrac{R}{\mm} \ar[rr]^{\cdot y_3}&& \dfrac{R}{(y_4,\ldots,y_h)} \ar[r] & \dfrac{R}{(y_3,\ldots,y_h)}  \ar[r] &  0\\
& \cdots  && \cdots  & \cdots & \\
0 \ar[r] & \dfrac{R}{\mm} \ar[rr]^{\cdot y_{h-1}}&& \dfrac{R}{(y_h)}  \ar[r] & \dfrac{R}{(y_{h-1},y_h)} \ar[r] &  0.
}
\end{displaymath}
Again by Step 2, we can apply \Cref{prop_ld_exactseq}(3) to these exact sequences. From the first sequence in the last display, the fact that $\lind_R (R/\mm) =\lind_R R/(y_2,\ldots,y_h)+1$, and \Cref{prop_ld_exactseq}(3), we get $\lind_R (R/\mm)=\lind_R R/(y_3,\ldots,y_h)$. From the remaining sequences of the last display, we get
\[
1\le \lind_R (R/\mm)=\lind_R R/(y_3,\ldots,y_h) =\cdots=\lind_R R/(y_h).
\]
From (iv) in Step 2, we get an exact sequence
\[
0 \longrightarrow \dfrac{R}{(y_1,\ldots,y_{h-1})} \xrightarrow{\cdot y_h} R \longrightarrow \dfrac{R}{(y_h)} \longrightarrow 0.
\]
Thanks to \Cref{prop_ld_exactseq}(2) and the fact that $\lind_R R/(y_h)\ge 1$, this yields
\[
\lind_R(R/\mm) =\lind_R R/(y_h) = \lind_R R/(y_1,\ldots,y_{h-1})+1=\lind_R (R/\mm)+1.
\]
The last equality follows from Equation \eqref{eq_lind_1}. This is in contradiction with $\lind_R k<\infty$. Therefore $R$ is Koszul. The proof of Theorem \ref{thm_HIy_artinGor} is completed.
\end{proof}
We finally get the
\begin{proof}[Proof of \Cref{thm_HIy_gstretched_long}]
(1) $\Longrightarrow$ (3): clear.

(3) $\Longrightarrow$ (2): If $\dim R=1$, $R$ is Koszul by \Cref{thm_Koszul}. If $\dim R=0$, it suffices to combine \Cref{prop_artin_nonGor_stretched} (the non-Gorenstein case) and \Cref{thm_HIy_artinGor} (the Gorenstein case).

(2) $\Longrightarrow$ (1): If $\dim R=1$, $R$ is absolutely Koszul by \Cref{prop_dim1_absKos}. If $\dim R=0$, it suffices to employ \Cref{prop_dim0_absKos}.

(2) $\Longleftrightarrow$ (4): by \Cref{thm_Koszul}. This concludes the proof.
\end{proof}

\section{\Cref{thm_main_HerzogIyengar} in the case $\dim R=1$ and in general}
\label{sect_dim1_case}
Our main result in the present work is
\begin{thm}
\label{thm_HerzogIyengar_AlmostMinMult}
Let $(R,\mm,k)$ be a Cohen--Macaulay local ring of almost minimal multiplicity. Assume that $\chara(k)=0$. If $\lind_R k<\infty$, then $R$ is Koszul.
\end{thm}

The main results of this section give the proof of \Cref{thm_HerzogIyengar_AlmostMinMult} in the case $\dim R=1$. To state the main results of this section, we fix the following assumptions. Let $(R,\mm,k)$ be a one dimensional Cohen--Macaulay local ring of almost minimal multiplicity. By \Cref{rem_complete}, we may assume that $R$ is complete so that $R=S/I$ with a regular local ring $(S,\nn)$ and an ideal $I\subset \nn^2$.

With this notation we have the following.
\begin{prop}
\label{prop_dim1_nonGor} 
Let $(S,\nn,k)$ be complete regular local ring, $I\subseteq \nn^2$ an ideal, and $(R,\mm)=(S/I,\nn/I)$. Assume that $(R,\mm,k)$ is a one dimensional Cohen--Macaulay local ring with almost minimal multiplicity, and that $\chara k=0$. Let $x\in \mm$ be a superficial element of $\mm$ with respect to $R$. Assume that $R/(x)$ is {\bf not} a Gorenstein ring, and that $\lind_R k<\infty.$ Then there is an equality $\lind_R k=0$, namely $R$ is a Koszul ring.
\end{prop}
\begin{proof}  Since $R$ has almost minimal multiplicity, so does $\overline{R}$, consequently $\overline{R}$ is a $g$-stretched artinian local ring. Let $\overline{\mm}=\mm/(x), h=\mu(\overline{\mm}), \tau=\text{type}(\overline{R})$ and let $s$ be the socle degree of $\overline{R}$. By the assumption that $R/xR$ is not Gorenstein, we have $\tau\ge 2$.  Consider the following two cases.

\textsf{Case 1:} $\tau=h$. Then, by  Proposition \ref{prop_EliasValla}, there is a minimal generating set $y_1,\ldots,y_h$ of $\ovl{\nn}:=\nn/(x)$ such that $\ovl{I}=(I+(x))/(x)\subseteq \ovl{S}$ is minimally generated by the elements $\{y_1 y_j\}_{2\le j\le h}$, $\{ y_i y_j\}_{2\le i\le j\le h}$, and $y_1^{s+1}$. An argument similar to the proof of Proposition \ref{prop_artin_nonGor_stretched} shows that 
\begin{itemize}
	\item[i)] $\dfrac{\mm}{(x)}=\dfrac{(y_2,x)}{(x)}\oplus \dfrac{(y_1,y_3,\ldots,y_h,x)}{(x)}$, and
	\item [ii)] $(x):y_2=\mm$.
\end{itemize}
In particular, we have
\begin{align*}
(y_2,x)\cap (y_1,y_3,\ldots,y_h,x)&=(x),\\
\mm & =(y_2,x)+(y_1,y_3,\ldots,y_h,x).
\end{align*}
Hence, we have the following exact sequences

$$0\to \frac{R}{(x)}\to \frac{R}{(y_2,x)}\oplus \frac{R}{(y_1,y_3,\ldots,y_h,x)}\to \frac{R}{\mm}\to 0 \ \ \ (\sharp_1)$$
and 
 \begin{displaymath}
	\xymatrix{  0  \ar[r] &  \dfrac{R}{\mm} \ar[r]^{\cdot y_2} & \dfrac{R}{(x)} \ar[r] & \dfrac{R}{(y_2,x)} \ar[r] & 0
	\ \ \  (\sharp_2)}.
\end{displaymath}
By way of contradiction, suppose that $\lind_R k\ge 1$, then by \Cref{rem_ldge1}, we get $\lind_R k\ge 2$.

Since $R$ is a 1-dimensional Cohen--Macaulay local ring and $x$ is regular on $R$, $\pd_R R/(x)=1$. In particular, $\Tor_{i}^R(k,R/(x))=0$ for all $i\ge 2$. On the other hand, we have $k\otimes_R R/(x)\cong k$ and $\Tor_{1}^R(k,R/(x))\cong \dfrac{(x)}{x\mm}$. Thus for $i=0,1$, the multiplication map $\Tor_{i}^R(k,R/\mm)   \xrightarrow{\cdot y_2}  \Tor_{i}^R(k,R/(x))$ is respectively $k \xrightarrow{\cdot y_2} k,$ and  $\mm/\mm^2 \xrightarrow{\cdot y_2} (x)/x\mm,$ which are both zero. (For the second map, note that $\mm y_2\subseteq (x)\cap \mm^2=\mm x$, as $x\in \mm\setminus \mm^2$.) Hence the injective map in the sequence $(\sharp_2)$ is Tor-vanishing. 

Therefore, by Proposition \ref{prop_ld_exactseq} and the fact that $\lind_R (R/\mm)\ge 2 > \lind_R R/(x)$, we conclude that $\lind_R R/(y_2,x)=\lind_R (R/\mm) +1.$ Now from the sequence $(\sharp_1)$ and $\pd_R R/(x)=1$, we have by \Cref{prop_ld_exactseq}(1) that 
\begin{align*}
\lind_R R/(y_2,x)&\le \lind_R \left(\frac{R}{(y_2,x)}\oplus \frac{R}{(y_1,y_3,\ldots,y_h,x)}\right)\\
&\le \max\{\pd_R R/(x)+1, \lind_R R/\mm\}= \max\{2, \lind_R R/\mm\}=\lind_R R/\mm.
\end{align*}
Hence $\lind_R k+1=\lind_R k$ which contradicts the finiteness of $\lind_R k$. Therefore $\lind_R k=0$, as desired.

\textsf{Case 2:} $\tau<h$. Also by Proposition \ref{prop_EliasValla}, we can find a minimal generating set $y_1,\ldots, y_h$ of $\nn$ such that $\ovl{I}\subseteq \ovl{S}$ is minimally generated by the elements $\{y_i y_j\}_{1\le i< j\le h}$, $\{y_j^2\}_{2\le j\le \tau}$, and $\{ y_i^2-u_iy_1^s\}_{2\le i\le h}$, where $u_i$ are units in $\ovl{S}$.	Similarly as in Case
1, we have
\begin{itemize}
	\item[i)] $\dfrac{\mm}{(x)}=\dfrac{(y_2,x)}{(x)}\oplus \dfrac{(y_1,y_3,\ldots,y_h,x)}{(x)}$.
	\item [ii)] $(x):y_2=\mm$.
\end{itemize}

Moreover, we deduce that if $\lind_R k<\infty$, then $R$ is Koszul, in the same manner as in Case 1. The proof is concluded.
\end{proof}

As in the artinian case, the harder part of the one-dimensional case of \Cref{thm_HerzogIyengar_AlmostMinMult} is when $R$ is Gorenstein.
\begin{thm}
\label{thm_dim1_Gor}
Let $(S,\nn,k)$ be complete regular local ring, $I\subseteq \nn^2$ an ideal, and $(R,\mm)=(S/I,\nn/I)$. Assume that $(R,\mm,k)$ is a one dimensional Cohen--Macaulay local ring with almost minimal multiplicity, and that $\chara k=0$. Let $x\in \mm$ be a superficial element of $\mm$ with respect to $R$. Assume that $R/(x)$ is a Gorenstein ring, and that $\lind_R k<\infty.$ Then there is an equality $\lind_R k=0$, namely $R$ is a Koszul ring.
\end{thm}

\begin{proof}
It is harmless to assume that $R$ is not regular. Assume by way of contradiction that $\lind_R k>0$, then by \Cref{rem_ldge1}, $\lind_R k\ge 2$.

Again let $\overline{\mm}=\mm/(x), h=\mu(\overline{\mm}), \tau=\text{type}(\overline{R})$ and $s$ be socle degree of $\overline{R}$; $s\ge 1$ since $R$ is not regular. Since $\ovl{R}$ is Gorenstein, we get that $\tau=1$. 

\textsf{Case 1:} $h=1$. Since $R=S/I$ and $x\in \mm\setminus \mm^2$, we may regard $x$ as an element of $\nn \setminus \nn^2$. Now $\ovl{S}=S/(x)$ is a regular local ring with a principal maximal ideal $\ovl{\nn}$. Let $y\in \nn\setminus \nn^2$ be such that $\ovl{y}$ be an uniformizer for $\ovl{\nn}$, then $\ovl{I}=(\ovl{y}^{s+1})$, since $s$ is the socle degree of $\ovl{R}$. By abuse of notation, we regard $y$ as an element of $R$; in particular, $\mm=(x,y)$.

Since $\ovl{R}=\ovl{S}/(y^{s+1})$, we have in $\ovl{R}$ the relations $(0):\ovl{y}=(\ovl{y}^s),$ $(0):\ovl{y^s}=(\ovl{y})$. Equivalently, in $R$,
\begin{align*}
(x):y &= (x,y^s),\\
(x):y^s &= (x,y)=\mm.
\end{align*}
Thus there are exact sequences
\begin{align}
0 & \to \dfrac{R}{(x,y^s)} \xrightarrow{\cdot y} \dfrac{R}{(x)} \to \dfrac{R}{\mm} \to 0, \label{eq_ses1}\\
0 & \to \dfrac{R}{\mm} \xrightarrow{\cdot y^s} \dfrac{R}{(x)} \to \dfrac{R}{(x,y^s)} \to 0. \label{eq_ses2}
\end{align}
We prove that for all $i\ge 0$, the injective (multiplication) maps in both exact sequences are Tor-vanishing.

Since $x$ is $R$-regular, $\Tor^R_i(k,R/(x))=0$ for all $i\ge 2$. Thus it suffices to consider $i=0,1$. For \eqref{eq_ses1}, applying $\Tor_0^R(k,-)$, we get
\[
\frac{R}{\mm} \xrightarrow{\cdot y} \frac{R}{\mm}
\]
which is zero. Since $\Tor^R_1(R/\mm,R/L)=L/(\mm L)$ for any $R$-ideal $L$, applying $\Tor_1^R(k,-)$ for \eqref{eq_ses1}, we get
\[
\frac{(x,y^s)}{\mm(x,y^s)} \xrightarrow{\cdot y} \frac{(x)}{\mm x}.
\]
This is again the zero map since $y^{s+1}=0$ in $R$. Hence \eqref{eq_ses1} induces the zero maps $\Tor^R_i(k,\cdot y)$. Similar arguments work for \eqref{eq_ses2}. 

Now we may apply \Cref{prop_ld_exactseq} for these two exact sequences. Since $2\le \lind_R k<\infty$, and $\lind_R (R/(x))\le \pd_R (R/(x))=1$, from \eqref{eq_ses1},
\begin{align*}
\lind_R k & \le \max\{\lind_R R/(x), \lind_R R/(x,y^s)+1\} \\
         &\le  \max\{1, \lind_R R/(x,y^s)+1\}=\lind_R R/(x,y^s)+1,\\
\lind_R R/(x,y^s) & \le \max\{\lind_R R/(x), \lind_R k-1\} \le \max\{1, \lind_R k-1\}=\lind_R k-1.
\end{align*}
Therefore $\lind_R R/(x,y^s)=\lind_R k-1$. Similarly, applying \Cref{prop_ld_exactseq} for \eqref{eq_ses2}, we deduce $\lind_R k= \lind_R R/(x,y^s)-1=\lind_R k-2.$
This contradiction shows that $\lind_R k=0$, i.e. $R$ is Koszul.

\textsf{Case 2:} $h\ge 2$. Then, by Proposition \ref{prop_EliasValla}, we can find a minimal generating set $\ovl{y_1},\ldots, \ovl{y_h}$ of $\ovl{\nn}=\nn/(x)$ such that $\ovl{I}=(I+(x))/(x) \subseteq \ovl{S}$ is minimally generated by the elements $\{\ovl{y_i}\ovl{y_j}\}_{1\le i< j\le h}$, and $\{\ovl{y_i}^2-\ovl{u_i}\,\ovl{y_1}^s\}_{2\le i\le h}$, where $\ovl{u_i}$ are units in $\ovl{S}$. 	
By the proof of \Cref{thm_HIy_artinGor}, we have a weak Koszul filtration for $\overline{R}$ given by
$$(0),\overline{\mm},\{(\bar y_1,\bar y_2,\ldots,\bar y_i)\}_{1\le i\le h-1}, \{(\bar y_j,\ldots,\bar y_h)\}_{2\le j\le h},$$
where $\bar y$ is the image of $y$ in $R/(x)$. Moreover, we have the following relations:
\begin{enumerate}[\quad \rm (a)]
 \item $(\ovl{y_1}):\ovl{y_i}=(\ovl{y_1},\ovl{y_2},\ldots,\ovl{y_{i-1}}):\ovl{y_i}=\ovl{\mm}$ for every $2\le i\le h$;
  \item $(\ovl{y_h}):\ovl{y_j}=(\ovl{y_{j+1}},\ovl{y_{j+2}},\ldots,\ovl{y_h}):\ovl{y_j}=\ovl{\mm}$ for every $2\le j\le h-1$;
 \item $(0):\ovl{y_1}=(\ovl{y_2},\ldots,\ovl{y_h})$;
 \item $(0):\ovl{y_h}=(\ovl{y_1},\ovl{y_2},\ldots,\ovl{y_{h-1}})$.
\end{enumerate}
The weak Koszul filtration lifting \Cref{lem_weakfiltr_lifting} implies that the collection $\Fc$ of ideals  
$$
(0),\mm, (x), (y_1,y_2,\ldots,y_i,x), (y_j,\ldots,y_h,x), \quad \text{where $1\le i\le h-1, 2\le j\le h$},
$$
forms a weak Koszul filtration for $R$. Moreover, the relations (a)--(d) translate to
\begin{enumerate}[\quad \rm (i)]
	\item $(y_1,x):y_i=(y_1,y_2,\ldots,y_{i-1},x):y_i=\mm$ for every $2\le i\le h$;
	\item $(y_h,x):y_j=(y_{j+1},y_{j+2},\ldots,y_h,x):y_j=\mm$ for every $2\le j\le h-1$;
	\item $(x):y_1=(y_2,\ldots,y_h,x)$;
	\item $(x):y_h=(y_1,y_2,\ldots,y_{h-1},x)$.
\end{enumerate}
From (i), we have the exact sequences
\begin{displaymath}
\xymatrixcolsep{6.72mm}
\xymatrixrowsep{4.8mm}
\xymatrix{
0  \ar[r] &  \dfrac{R}{\mm} \ar[rr]^{\cdot y_h} && \dfrac{R}{(y_1,\ldots,y_{h-1},x)} \ar[r] & \dfrac{R}{\mm} \ar[r] & 0 \quad (\text{I.}h)\\ 
0  \ar[r] & \dfrac{R}{\mm} \ar[rr]^{\cdot y_{h-1}}&& \dfrac{R}{(y_1,\ldots,y_{h-2},x)} \ar[r] & \dfrac{R}{(y_1,\ldots,y_{h-1},x)}  \ar[r] &  0 \quad  (\text{I.}(h-1))\\
& \cdots  && \cdots  & \cdots & \\
0 \ar[r] & \dfrac{R}{\mm} \ar[rr]^{\cdot y_2}&& \dfrac{R}{(y_1,x)}  \ar[r] & \dfrac{R}{(y_1,y_2,x)} \ar[r] &  0. \quad (\text{I.}2)
}
\end{displaymath}

From (iii), we get the exact sequence
\[
0 \longrightarrow \dfrac{R}{(y_2,\ldots,y_h,x)} \xrightarrow{\cdot y_1} \dfrac{R}{(x)} \longrightarrow \dfrac{R}{(y_1,x)} \longrightarrow 0. \quad (\text{III})
\]
From (ii), we get the exact sequences
\begin{displaymath}
\xymatrixcolsep{6.72mm}
\xymatrixrowsep{4.8mm}
\xymatrix{
0  \ar[r] &  \dfrac{R}{\mm} \ar[rr]^{\cdot y_2} && \dfrac{R}{(y_3,\ldots,y_h,x)} \ar[r] & \dfrac{R}{(y_2,\ldots,y_h,x)} \ar[r] & 0\quad (\text{II.2})\\ 
0  \ar[r] & \dfrac{R}{\mm} \ar[rr]^{\cdot y_3}&& \dfrac{R}{(y_4,\ldots,y_h,x)} \ar[r] & \dfrac{R}{(y_3,\ldots,y_h,x)}  \ar[r] &  0\quad (\text{II.3})\\
& \cdots  && \cdots  & \cdots & \\
0 \ar[r] & \dfrac{R}{\mm} \ar[rr]^{\cdot y_{h-1}}&& \dfrac{R}{(y_h,x)}  \ar[r] & \dfrac{R}{(y_{h-1},y_h,x)} \ar[r] &  0. \quad (\text{II.}(h-1))
}
\end{displaymath}
From (iv), we get an exact sequence
\[
0 \longrightarrow \dfrac{R}{(y_1,\ldots,y_{h-1},x)} \xrightarrow{\cdot y_h} \frac{R}{(x)} \longrightarrow \dfrac{R}{(y_h,x)} \longrightarrow 0. \quad (\text{IV})
\]
By \Cref{lem_weakfiltr_Tormap}, the injective (multiplication) maps of the exact sequences (I.$i$) (where $2\le i\le h$), (II.$j$) (where $2\le j\le h-1$), (III), (IV), are all Tor-vanishing.

From (I.$h$), \Cref{prop_ld_exactseq} and the fact that $2\le \lind_R (R/\mm) < \infty$, we get
\[
\lind_R R/(y_1,\ldots,y_{h-1},x)=\lind_R (R/\mm).
\]
Similarly, considering the exact sequences (I.$i$) where $2\le i\le h-1$, we deduce
\[
\lind_R (R/\mm)=\lind_R R/(y_1,\ldots,y_{h-1},x)=\cdots=\lind_R R/(y_1,y_2,x)=\lind_R R/(y_1,x).
\]
Now $\lind_R R/(y_1,x)\ge 2$ and $\lind_R R/(x)\le \pd_R R/(x)=1$. Therefore the exact sequence (III) and  \Cref{prop_ld_exactseq} imply that
$$
\lind_R R/(y_2,\ldots, y_h,x) = \lind_R R/(y_1,x)-1=\lind_R (R/\mm)-1.
$$
Arguing similarly, the exact sequences (II.$j$) (where $2\le j\le h-1$) imply that
$$1\le \lind_R(R/\mm) =\lind_R R/(y_3,\ldots, y_h,x) =\cdots= \lind_R R/(y_h,x).$$
Finally, the exact sequence (IV) and  \Cref{prop_ld_exactseq} imply that
\begin{align*}
\lind_R (R/\mm) &= \lind_R R/(y_1,\ldots, y_{h-1},x) \le \max\{\lind_R R/(x),\lind_R R/(y_h,x)-1\}\\
               & = \max\{\lind_R R/(x),\lind_R (R/\mm)-1\} \\
               & \le \max\{1,\lind_R (R/\mm)-1\}=\lind_R (R/\mm)-1.
\end{align*}
This contradiction implies that $\lind_R k=0$, i.e. $R$ is Koszul. The proof is concluded.
\end{proof}

Finally, we present the proof of our main result.

\begin{proof}[Proof of \Cref{thm_HerzogIyengar_AlmostMinMult}]
Again using \Cref{prop.passtocompletion} and \Cref{rem_complete}, we may assume that $R=S/I$, where $(S,\nn,k)$ is a complete regular local ring, $I\subseteq \nn^2$. By \Cref{prop_dim1_reduction}, it suffices to treat the case $\dim R\le 1$.

\textsf{Case 1:} $\dim R=0$. The desired conclusion follows from ``(3) $\Longrightarrow$ (2)" in Theorem \ref{thm_HIy_gstretched_long}.

\textsf{Case 2:} $\dim R=1$. The desired conclusion follows from Proposition \ref{prop_dim1_nonGor} and Theorem \ref{thm_dim1_Gor}. 
\end{proof}

\section{Questions and remarks}

In view of \Cref{thm_HIy_gstretched_long} and its proof, it is natural to ask the following questions.
\begin{quest}
Let $(R,\mm,k)$ be a $g$-stretched, artinian local ring with $\chara(k)=p>0$. Is it true that if $\lind_R k<\infty$, then $R$ is Koszul?
\end{quest}
\begin{quest}
\label{quest_absK}
Let $(R,\mm)$ be an absolutely Koszul local ring such that $\mu(\mm^2)\le 3$. Is it true that $R$ is a Koszul ring?
\end{quest}
If $(R,\mm)$ is a standard graded $k$-algebra with $\mu(\mm^2)\le 3$, thanks to works of D'Al\`i \cite{D} and Ahangari-Maleki and \c{S}ega \cite{AMS20}, we have the remarkable equivalence that $R$ is Koszul if and only if it is absolutely Koszul; see \cite[Main Theorem]{AMS20}.

As mentioned before, Conca, De Negri and Rossi \cite[Proposition 6.1.8]{CDR} showed that if $R$ is Cohen--Macaulay of almost minimal multiplicity, with the type $r(R)\le \codim(R)-1$, then $R$ is Koszul. As a corollary of \Cref{thm_Koszul}, we prove that the converse of their result holds in the graded case.
\begin{cor}
\label{cor_almostminmult_Koszul}
Let $(R,\mm)$ be a standard graded $k$-algebra. Assume that $R$ is Cohen--Macaulay of almost minimal multiplicity, namely $e(R)=\codim(R)+2$. Then the following are equivalent:
\begin{enumerate}[\quad \rm (1)]
\item $R$ is Koszul;
\item $r(R)\le \codim(R)-1$.
\end{enumerate}
\end{cor}
\begin{proof}
For (2) $\Longrightarrow$ (1), we may use the argument of \cite[Proposition 6.1.8]{CDR}. Conversely, assume that $R$ is Koszul. Let $\ovl{k}$ be the algebraic closure of $k$. Passing from $R$ to $R\ot_k \ovl{k}$ does not affect the statement, so we may assume that $k$ is infinite.

Now we reduce to the artinian case. Assume that $\dim(R)\ge 1$. Since $R$ is a Cohen--Macaulay and $k$ is infinite, we may choose a regular $R$-sequence $\bsx=x_1,\ldots,x_{\dim(R)}$ consisting of linear forms. Passing from $R$ to $\ovl{R}=R/(\bsx)R$ does not affect the statement: the Koszul property, the type, and codimension of $R$ are not affected by such an operation, as $\bsx$ is a regular sequence of linear forms. Hence we may assume that $R$ is artinian. That $\ell(R)=\mu(\mm)+2$ implies that $R$ is artinian $g$-stretched. The argument of \Cref{thm_Koszul}, (1) $\Longrightarrow$ (2) goes through in this case, and therefore $r(R)\le \codim(R)-1$, as desired.
\end{proof}

We have no counterexample to the following local version of \Cref{cor_almostminmult_Koszul}.
\begin{quest}
Let $(R,\mm)$ be a Cohen--Macaulay local ring of almost minimal multiplicity. If $R$ is Koszul, does it follow that $r(R)\le \codim(R)-1$?
\end{quest}

\section*{Acknowledgments}

We are grateful to Naoyuki Matsuoka for suggesting the statement of \Cref{thm_Matsuoka_intro}, and Marilina Rossi and Joan Elias for their interest and fruitful discussions related to the content of this paper. Both authors are grateful to the support of the Vietnam Academy of Science and Technology under grant number CTTH00.01/25-26. The second author (HDN) was partially supported by NAFOSTED under the grant number 101.04-2023.30.

\end{document}